\documentclass[12pt,a4paper]{amsart}
\usepackage[dvipsnames]{xcolor}
\usepackage{amsfonts}
\usepackage{amscd}
\usepackage[latin1]{inputenc}
\usepackage{t1enc}
\usepackage[mathscr]{eucal}
\usepackage{indentfirst}
\usepackage{graphicx}
\usepackage{graphics}
\usepackage{pict2e}
\usepackage{epic}
\numberwithin{equation}{section}
\usepackage[margin=2.6cm]{geometry}
\usepackage{epstopdf} 
\usepackage{amsmath}
\allowdisplaybreaks
\usepackage{amssymb}
\usepackage{amsthm}
\usepackage{bbm}
\usepackage{tikz}
\usepackage{paralist}
\usepackage[onehalfspacing]{setspace}
\usepackage{amsfonts}
\usepackage{color}
\usepackage{nicefrac}
\usepackage{mathrsfs}
\usepackage{multirow}
\usepackage{mathtools}
\usepackage{tikz-cd}
\usepackage{stix}
\newcounter{dummy}
\usepackage{hyperref}
\usepackage{enumitem}
\makeatletter
\newcommand\myitem[1][]{\item[#1]\refstepcounter{dummy}\def\@currentlabel{#1}}
\makeatother
\newtheorem{thm}{Theorem}
\numberwithin{thm}{section}
\newtheorem{lemma}[thm]{Lemma}

\newtheorem{definition}[thm]{Definition}
\newtheorem{coro}[thm]{Corollary}

\newtheorem*{thm*}{Theorem}
\newtheorem*{prop*}{Proposition}
\numberwithin{equation}{section}
\theoremstyle{remark}
\newtheorem{remark}[thm]{Remark}

\newcommand{\A}{\mathscr{A}}

\newcommand{\R}{\mathbb{R}}
\newcommand{\N}{\mathbb{N}}

\newcommand{\T}{\mathbb{T}}

\newcommand{\LL}{\mathcal{L}} 

\newcommand{\Ncal}{\mathscr{N}}

\renewcommand{\phi}{\varphi}

\newcommand{\dx}{\,\textup{d}x}
\newcommand{\dt}{\,\textup{d}t}
\newcommand{\dy}{\,\textup{d}y}
\newcommand{\dz}{\,\textup{d}z}
\newcommand{\dlambda}{\,\textup{d}\lambda}

\newcommand{\dH}{\,\textup{d}\mathcal{H}}

\everymath{\displaystyle}

\DeclareMathOperator{\Lin}{Lin}

\DeclareMathOperator{\divergence}{div}

\renewcommand{\phi}{\varphi}

\DeclareMathOperator{\id}{id}

\newcommand{\M}{\mathscr{M}}

\definecolor{Gump}{rgb}{0,0.6,0.4}
\definecolor{Hanks}{rgb}{0.7,0.3,0.1}

\begin{document}
\title[Higher integrability under additional sign constraints]{Higher integrability of solutions to elliptic equations under additional sign constraints}
\author[Schiffer]{Schiffer, Stefan}
\address{parcIT GmbH, Erfstra{\ss}e 15, 50672 Cologne, Germany} 
\email{stefan.schiffer.math@gmail.com}
\begin{abstract}
 Solutions to elliptic equations often exhibit higher regularity properties such as \emph{higher integrability}. That is, for instance, a solution $u$ to a system that a priori only satisfies $ u \in W^{1,r}$ is more regular and even in the Sobolev space $W^{1,s}$ for some $s>r$. Under additional constraints of the sign of specific terms such as $(\partial_i u)$ this improvement of regularity can be sharpened further.

 In this work, we consider two examples of such higher integrability results: First, we show a version of M\"uller's result on the higher integrability of the determinant for maps $u \in W^{1,n} $ such that $\mathrm{det}(\nabla u) \geq 0$ (or $ \mathrm{det}_-(\nabla u) \in L \log L$). Second, we consider (very weak) solutions to the $p$-Laplace equation that satisfy sign constraints for their partial derivatives, i.e. that $(\partial_i u)_- $ is of higher integrability than $(\partial_i u)_+$. To prove our results, we use the method of Lipschitz truncation; for the second example we further develop a variation of this technique, the \emph{asymmetric} Lipschitz truncation.
\end{abstract}
\maketitle

\section{Introduction}
In this work we consider higher integrability properties of Sobolev function that are solutions to certain \emph{elliptic} equations or, in more generality, solutions to differential inclusions $\nabla u \in K$. As an example, consider the linear elliptic equation 
\[
\divergence( A(x) \cdot \nabla u(x)) = 0
\]
where $A \colon \Omega \to \Lin(\R^n;\R^n)$ is positive definite almost everywhere. The natural space to consider weak solutions to this system is the Sobolev space $W^{1,2} = \{ u \in L^2 \colon \nabla u \in L^2\}$ i.e. 
\begin{equation} \label{intro:weaksolution}
u \in W^{1,2} \quad \text{and} \int_{\Omega} \nabla \varphi(x) \cdot (A(x) \nabla u(x)) =0 \quad \forall \varphi \in W^{1,2}_0(\Omega).
\end{equation}
Regularity theory for such an elliptic equation nowadays is quite classical, in particular when the coefficients $A$ are smooth or at least H\"older continuous. But even for \emph{non-continuous} coefficients $A$ one can show that solutions to the in the energy space $W^{1,2}$ are locally in some space of better integrability, i.e. $W^{1,2+\varepsilon}$. A \emph{very weak} solution to this system is a solution that is not in this energy space, e.g. in $W^{1,r}$ for $r<2$, where we need to upgrade the regularity of the test function to $W^{1,\infty}_0$ (or $C_c^{\infty}$). One can now show that for sufficiently irregular $A$ (that are still elliptic):
\begin{itemize}
    \item If $r$ is sufficiently close to $2$, any very weak solution in $W^{1,r}$ is already a weak solution, i.e. in $W^{1,2}$ (e.g. \cite{Meyers}).
    \item If $r$ is below a certain threshold, then there might exist a weak solution $u \in W^{1,r}$ which is not in $W^{1,2}$ (e.g. \cite{AFS}).
\end{itemize}
In this work, we add an additional sign constraint of involved quantities. In above example that might look like 
\[
(\partial_{i} u)_+ \in L^r \quad \text{and} \quad (\partial_{i} u)_- \in L^2,
\]
i.e. the negative part of the partial derivative exhibits far better integrability than the positive part. The aim of this work is to show that for two examples, such 'one-sided better integrability' is already enough to infer it for the whole integrand.
\subsection{Higher integrability of the determinant}
The first example builds on a well-known result by \textsc{M\"uller} \cite{Mueller,Mueller2}. 
Recall that a map $f \colon \Omega \subset \R^n \to \R^n$ is called $K$-quasiconformal or $K$-quasiregular if for almost every $x \in \Omega$ the norm of the derivative is controlled by its determinant, i.e.
\begin{equation} \label{quasiconformal}
K^{-1} \vert \nabla u (x) \vert^n \leq \det( \nabla u).
\end{equation}
As this condition is homogeneous of order $n$, the natural space to consider those mappings is $W^{1,n}$. As it was already shown by \textsc{Gehring} (\cite{Gehring,GR66} , see also \cite{Astala}), maps satisfying \eqref{quasiconformal} exhibit higher regularity properties, i.e. both
\begin{equation} \label{regularity1}
    u \in W^{1,n} \text{ and satisfies \eqref{quasiconformal}} \quad \Longrightarrow \quad u \in W^{1,n+\varepsilon}
\end{equation}
and 
\begin{equation} \label{regularity2}
    u \in W^{1,n-\varepsilon} \text{ and satisfies \eqref{quasiconformal}} \quad \Longrightarrow \quad u \in W^{1,n}.
\end{equation}
Of course we may directly combine \eqref{regularity1} and \eqref{regularity2}, but it is also sensible to split those results: The first one tells us something about the regularity of maps that are already in the right \emph{energy space} (i.e. where $\det(\nabla u) \in L^1$), whereas the second improves solutions that are \emph{not} in this space of \emph{classical} solutions. It is also worth pointing out that, in space dimension 2, such maps  are very much connected to problem \eqref{intro:weaksolution} in the following way (cf. \cite{AFS}): We have 
\begin{equation*}
     \divergence \sigma = 0 \quad \text{for} \quad \sigma= A(x) \cdot \nabla u(x)
\end{equation*}
In two dimensions, that means that $\sigma = \nabla^\perp w$ for some function $w$. Then $\bar{u}=(u,w)$ is a quasiregular map (with $K$ depending on the properties of the operator $A$).

A small part of aforementioned improved integrability property survives even if the parameter $K$ diverges to infinity, i.e. we are left with
\[
0 \leq \det(\nabla u).
\]
Indeed, if $u \in W^{1,n}_{loc}$ and $\det$ is positive \textsc{M\"uller} showed that $\det(\nabla u) \log(1+\vert \nabla u\vert) \in L^1_{loc}$ (also cf. \cite{Cianchi,Greco2,Greco1,IV} for refinements). The limiting result of \eqref{regularity2} was then also proven by \textsc{Iwaniec \&  Sbordone}, i.e. if $\vert \nabla u \vert^n \log(1+ \vert \nabla u \vert)^{-1} \in L^1$ and $\det$ is (locally) positive, then $\det(\nabla u) \in L^1$ (cf. \cite{IS2}).

Recently, \textsc{Rai\c{t}\u{a}} \cite{Raita25} isolated  showed that quasiconcavity of the determinant (it is even quasiconcave \emph{and} quasiconvex, i.e. a Null-Lagrangian) is responsible for this improved regularity. Recall that $F$ is quasiconcave if
\[
F(v) \geq \int_{\T_n} F(v + \nabla \psi(x)) \dx 
\]
for any periodic function $\psi \in C^{\infty}(\T_n,\R^n)$. One then can show that 
if $F$ is $p$-homogeneous and quasiconcave then
\[
u \in W^{1,p} \quad \text{and} \quad F(\nabla u) \geq 0 \quad \Longrightarrow \quad F(\nabla u) \log(1+F(\nabla u)) \in L^1.
\]

The purpose of considering this first example in the present note is three-fold:
\begin{enumerate} [label=(\roman*)]
    \item We prove the analaogue of the Iwaniec-Sbordone result for quasiconcave functions (cf. Theorem \ref{thm:main2} below).
    \item In the setting of \cite{Raita25} we prove some minor refinements, i.e. that if $\vert \nabla u \vert^p \log(1+ \vert \nabla u )^{\alpha} \in L^1$ and $F_-$ is higher integrable, then $\vert \nabla u \vert^p \log(1+ F(\nabla u) )^{\alpha+1}$.
    \item As a method of proof we show that the method of \emph{Lipschitz truncation}, that can for instance be used to obtain \eqref{regularity1} and \eqref{regularity2}, also works in this limiting case $K \to \infty$. The key idea of this method is, instead of directly considering $u$, to mainly study a modified version of $u$ (which we call $u_{\lambda}$) that is already in $W^{1,\infty}$.
\end{enumerate}
In more detail, we assume that 
\begin{itemize}
    \item $F$ is a quasiconcave function with $p$-growth, i.e.
        \[
        \vert F(v) \vert \leq C(1+\vert v\vert^p);
        \]
    \item $F(0) =0$;
    \item $u \in L^1(\R^n,\R^n)$ has compact support.
\end{itemize}

Then we obtain the following results:

\begin{thm} \label{thm:main1}
    Suppose that $\alpha \geq -1$ and that
    \begin{enumerate} [label=(\roman*)]
        \item $\vert \nabla u \vert^p \log(1+\vert \nabla u \vert)^{\alpha} \in L^1$;
        \item $F_-(\nabla u) \cdot \log(1+ \vert F_-(\nabla u) \vert)^{\alpha+1} \in L^1$.
    \end{enumerate}
    Then $F_+(\nabla u) \cdot \log(1+ F_+(u))^{\alpha+1} \in L^1$.
\end{thm}

\begin{thm} \label{thm:main2}
    Suppose that $\alpha <-1$ and that
    \begin{enumerate} [label=(\roman*)]
        \item $\vert \nabla u \vert^p \log(1+\vert \nabla u \vert)^{\alpha} \in L^1$;
        \item $F_-(\nabla u) \cdot \log(1+ \vert F_-(\nabla u) \vert)^{\alpha+1} \in L^1$.
    \end{enumerate}
    Then $F_+(\nabla u) \cdot \log(1+ F_+(\nabla u))^{\alpha+1} \in L^1$.
\end{thm}
There are two main reasons why the result is exactly split in the fashion above. First, the proofs of the cases $\alpha \geq -1$ and $\alpha <-1$ are different, which can be (at least partially) attributed to the integrability of $(1+\lambda)^{-1} \log(1+\lambda)^{\alpha}$ on $(1,\infty)$. Second, in the first result we can guarantee that $F_+(\nabla u) \in L^1$ (i.e. $F$ is integrable) whereas in the second result this cannot be achieved. We also note the following:
\begin{enumerate} [label=(\roman*)]
\item As mentioned before, in the scope of the determinant, much more general results have been proven (cf. for instance Table 1 in \cite{Greco1} for an overview) and in general the proofs for quasiconcave functions should work similarly. The present work therefore rather provides a different method in showing those result (which surely can be generalised to more general Orlicz functions).
    \item As long as an $\A$-free truncation is available, similar to \cite{Raita25}, the results can easily be generalised from the case of quasiconcavity to $\A$-quasiconcavity.
    \item Local results (i.e. replacing the assertion that $\vert \nabla u \vert^p \log(1+\vert \nabla u \vert)^{\alpha} \in L^1$ by $\in L^1_{loc}$) can be usually be recovered by considering
    \[
    \tilde{u} = \varphi u, \quad \nabla \tilde{u} = \varphi \nabla u + \nabla \varphi \otimes u.
    \]
    Observe that we then need to verify the second assertion on the integrability of the negative part. Due to Sobolev embedding and local Lipschitz continuity of $F$, the second term is of high integrability and thus only the first is critical. For this term we could for instance assume that $F$ is $p$-homogeneous.
    \item In principle one can also generalise the results to functions with Orlicz growth, i.e. $\vert F(v) \vert \leq C(1+ \Psi(Cv))$ for suitable $\Psi$. We do not consider this case, however: If $\Psi \sim \vert v \vert^p \log(1+ \vert v \vert)^{\alpha}$, proofs are rather similar. Furthermore, under our assumptions we cannot explore any situation with any critical growth like $\Psi \sim \vert v \vert \log(1+\vert v \vert)$: Even in the case where $\Psi \leq C \vert v \vert^{2-\varepsilon}$ (i.e. in the case of subquadratic growth) the growth condition already dictates that $F_+(\vert v \vert)$ can only grow linearly and we directly have much higher integrability of $F_+$. Similarly, if $\Psi$ has super-polynomial growth, $F_+$ also instantly must have slower growth than $F_-$.
    \end{enumerate}
\subsection{Higher integrability of certain elliptic systems} \label{intro:part2}
Let $1<p< \infty$ and consider solutions $u \in W^{1,q}(\R^n;\R)$ to the system
\[
\divergence( A(x,\nabla u))  = \divergence f.
\]
The operator $A$ may or may not be linear, we restrict ourselves to A of the form
\begin{equation} \label{eq:form:A}
    A(x,\nabla u(x)) = a_1 (\vert \nabla u \vert) (a_2(x) \nabla u(x))
\end{equation}
where 
\begin{enumerate} [label=(A\arabic*)]
    \item \label{item:A1} $a_1 \colon (0,\infty) \to (0,\infty)$ is such that
        \[
            c t^{p-2}\leq a_1(t) \leq C t^{p-2}, t>1 \quad \text{ and } a_1(t) \leq C t^{-1+\varepsilon}, 0<t<1;
        \]
    \item \label{item:A2} The measurable map $a_2 \colon \Omega \to \Lin(\R^n;\R^n)$ satisfies:
        \begin{itemize}
            \item $a_2$ maps into the space of positive definite diagonal matrices;
            \item $a_2(x) \geq \id$ and there exists a constant $\nu \geq 1$ such that $a_2(x) \leq \nu \id$ (in the sense of symmetric matrices).
        \end{itemize}
\end{enumerate}
Note that for $a_1(\cdot) =1 $ these assumptions cover some systems of the general form \eqref{intro:weaksolution} and for $a_1(t) = t^{p-2}$ also includes systems of $p$-Laplacian type.

In the case of the $p$-Laplacian an intriguing dichotomy arises: For any $p \neq 2$ there is an $\varepsilon(p)$ such that
\begin{itemize}
    \item for $\varepsilon < \varepsilon(p)$ any solution $u \in W^{1,p-\varepsilon}$ improves to $u \in W^{1,p}$ (cf. \cite{IS,Lewis});
    \item for $\varepsilon > \varepsilon(p)$ this is not true any more \cite{CT,KMSX}.
\end{itemize}
The latter estimate is shown using \emph{convex} integration and, in particular, the method of staircase laminates developed by \textsc{Faraco} \cite{AFS,Faraco}. \textsc{Colombo \& Tione} \cite{CT} show even a little bit more: If one adds the constraint
\[
(\partial_i u)_- \in L^{\infty}.
\]
then the higher integrability property of the $p$-Laplacian extends to the full range, i.e. if
\[
(\partial_i u)_+ \in L^r,~r\geq \max \{1,p-1\} \quad \text{and} \quad (\partial_i u)_- \in L^{\infty}
\]
for a very weak solution to the $p$-Laplace equation, then $u \in W^{1,p}$. In the present work, we show a strengthening of this result. 
\begin{thm} \label{thm:main3}
    Suppose that $u \in W^{1,r}_0(\Omega;\R)$ is a solution to the equation
    \begin{equation}
            \divergence(A(x,\nabla u)) = \divergence f.
    \end{equation}
    where $A$ obeys the properties \ref{item:A1} and \ref{item:A2}, $\Omega \subset \R^n$ is a convex domain, $f \in (L^1 \cap L^{\infty})(\R^n;\R^n)$ and $\max \{1,p-1\} < r < p$. Suppose further that
    \[
    (\partial_i u)_- \in L^q \quad \text{for all } i=1,\ldots,n, \quad r<q<p.
    \]
    Then 
    \[
    \nabla u \in L^{q-\delta} \quad \text{for any } \delta >0.
    \]
\end{thm}
Of course, it should be mentioned that (by a reflection argument) a similar statement is true for any combination of partial derivatives so that either $(\partial_i u)_+$ or $(\partial_i u)_-$ has higher integrability.

The main restrictions of above theorem lie in the fact that $u$ has zero boundary values on a convex domain. Under specific restrictions on the exponents above results can obviously expanded to  a local setting when considering $\tilde{u} = \varphi u$ or by testing the equation with $\varphi u_{\lambda}$ (and not the truncated version $u_{\lambda}$ direclty). The assumption of convexity (which only enters in at one point) can likely be dropped.

Coming back to the observation of \textsc{Colombo \& Tione}: Theorem \ref{thm:main3} also entails the following.

\begin{coro}  \label{coro:main4}
    Suppose that $\Omega$ is a convex domain. Further suppose that $\varepsilon_0=\varepsilon_0(p,\nu)$ is such that for any solution $u \in W^{1,1}_0(\Omega)$ to the equation
    \begin{equation*}
            \divergence(A(x,\nabla u)) = \divergence f
    \end{equation*}
    we have for any $\varepsilon< \varepsilon_0$
    \[
    \nabla u \in L^{p-\varepsilon} \quad \Rightarrow \quad \nabla u \in L^p.
    \]
    Then this 'higher integrability' property also holds true if we just assume
    \[
    (\partial_i u)_+ \in L^r,~r > \max \{p-1,1\} \quad \text{and} \quad (\partial_i u)_- \in L^{p-\varepsilon} \quad \text{for all } i=1,\ldots,n.
    \]
\end{coro}

In other words, the 'positive' result of \textsc{Colombo \& Tione} \cite{CT} entails that the staircase laminates construction breaks down, if the staircase construction is bounded to one side; whereas Theorem \ref{thm:main3} and Corollary \ref{coro:main4} tells us that a even a truly asymmetric staircase construction (in the sense that one sign is of completely different integrability than the other sign) is, at least for $p$-Laplace type systems, impossible (cf. Figure \ref{fig1}).
\begin{figure} \label{fig1}
    \centering
    \includegraphics[scale=0.2]{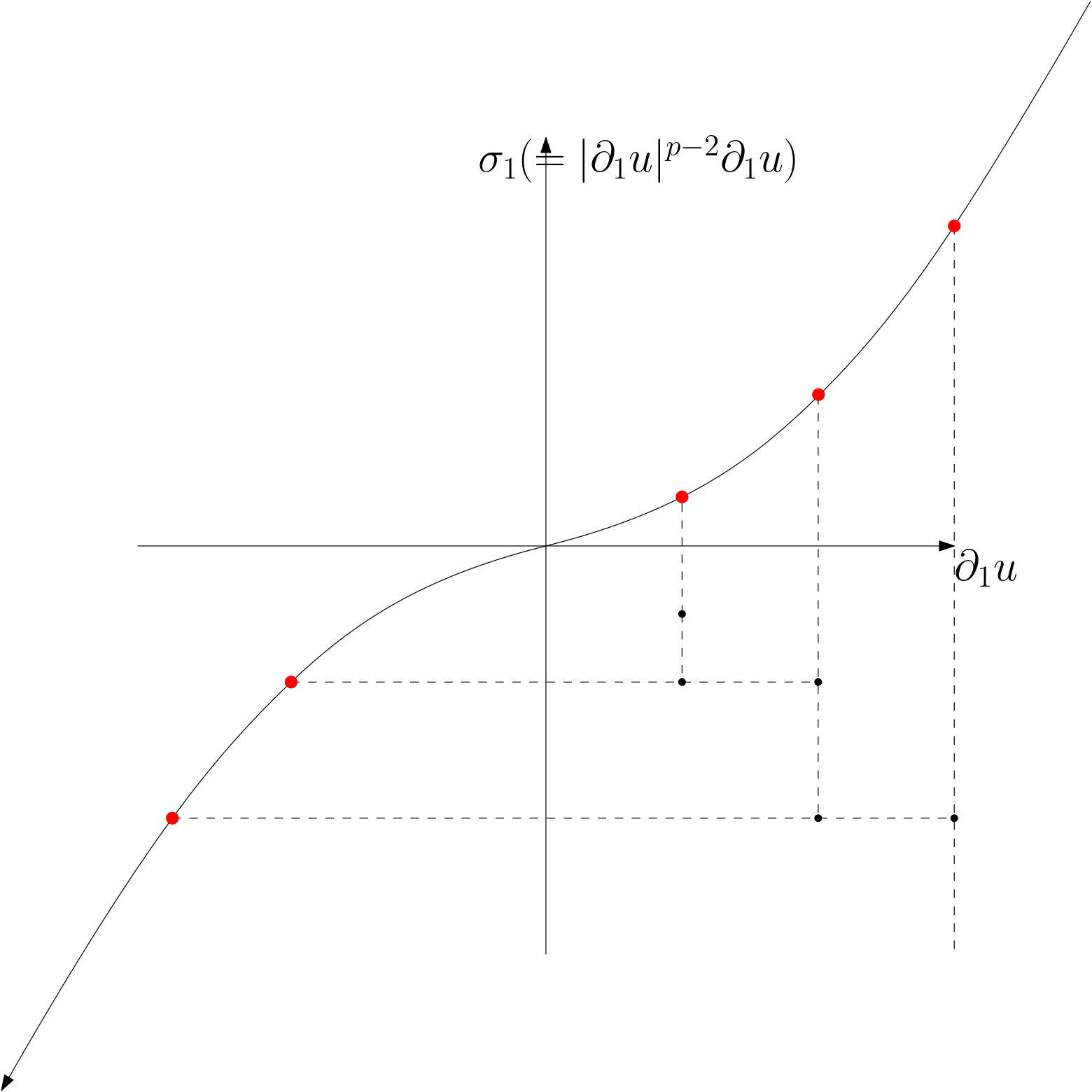}\caption{Rough outline of the start of the staircase construction: The elliptic equation is seen as a differential inclusion, i.e as a pair $(\nabla u,\sigma)$ where $\divergence \sigma=0$ and $\sigma= \vert \nabla u \vert^{p-2} \nabla u$. At the end of the construction (up to some small error), we obtain a pair $(\nabla u,\sigma)$ such that $(\nabla u(x),\sigma(x))$ is contained in the red points almost everywhere. For this staircase construction (which is based on rank-one connections) to work it is crucial that $\partial_1 u$ is allowed to have (unbounded) positive and negative values. }
\end{figure}

It should be mentioned that the $L^{\infty}$-version of Corollary \ref{coro:main4} \emph{includes} the exponent $r = \max\{1,p-1\}$. If $p <2$, the method chosen in this work is (so far) unable to reach the exponent $r=1$; this fact should probably attributed to details in the proof and it should be expected that the result of Corollary \ref{coro:main4} is also valid for $r=1$.
\subsection{Structure}
We give a short outline of the remainder of this article. In Section \ref{sec:prelim}, we first remind the reader of some (by now) standard results on Lipschitz truncation. We then show a modification of those results that is tailored to our framework. The 'standard' Lipschitz truncation results are then used to show Theorems \ref{thm:main1} and \ref{thm:main2} in Section \ref{sec:determinant}. Finally, in Section \ref{sec:pLaplace} we show Theorem \ref{thm:main3}.
\subsection*{Acknowledgements} Parts of this paper were written during a research stay at the Max-Planck-Institute for Mathematics in the Sciences in Leipzig and the author thanks the institute for its kind hospitality.
\subsection*{Notation} We denote constants by $C$. The value of $C$ may change from line to line. We for instance write $C(\varepsilon)$ or, if the value of the constant depends on $\varepsilon$.

\section{Preliminaries} \label{sec:prelim}
\subsection{Standard Lipschitz truncation}
We shortly remind the reader of some crucial notions. Recall that the (Hardy-Littlewood) maximal function is defined via
\[
\M v(x) := \sup_{r>0} \fint_{B_r(x)} \vert v \vert(y) \dy.
\]
It is well-known that this operator $\M$ is sublinear, bounded from $L^{\infty}$ to $L^{\infty}$ and from $L^1$ to the weak space $L^{1,\infty}$, meaning that 
\begin{equation} \label{eq:max:1}
    \LL^n(\{ \M v \geq \lambda \}) \leq \frac{C_{\M}}{\lambda} \int \vert v \vert \dx.
\end{equation}

The following result (e.g. \cite{AF84,Zhang}) relies on two crucial observations:
\begin{enumerate} [label=(\roman*)]
    \item a weakly differentiable function is Lipschitz continuous on sublevel sets of its maximal function (cf. \cite{AF84,Liu});
    \item we can extend any Lipschitz function on a set $X \subset \R^n$ to the full space (\cite{Kirszbraun,Whitney}).
\end{enumerate}

\begin{thm}[Lipschitz truncation] \label{thm:LT}
    Suppose that $u \in W^{1,1}(\R^n;\R^m)$.
  For any $\lambda>0$ there exists a function $u_{\lambda} \in W^{1,\infty}(\R^n;\R^m)$ such that 
        \[
        \Vert \nabla u_{\lambda} \Vert_{L^{\infty}} \leq c'(n) \lambda
        \]
        and 
        \[
        \{ u \neq u_{\lambda} \} \subset \{\M
        (\nabla u) > \lambda\}.
        \] 
\end{thm}
Observe that if $u$ has compact support, so does $u_{\lambda}$:  Outside of a possibly big ball around the support, $\M (\nabla u) \leq \lambda$. Further recall the following estimate of \cite{Zhang} that directly gives 
\begin{lemma} \label{lemma:Zhang}
Let $\lambda>0$ and $v \in L^1(\R^n;\R^m)$. Then 
\begin{equation} \label{eq:max:3}
    \LL^n(\{ \M v > \lambda\}) \leq \frac{c}{\lambda} \int_{\{\vert v \vert > \lambda/2\}} \vert v \vert \dx.
\end{equation}
\end{lemma}

\subsection{Asymmetric Lipschitz truncation}
The constrained higher integrability of elliptic system \eqref{intro:part2} needs a type of Lipschitz truncation that respects that positive and negative part of partial derivatives cannot be treated equally. For this we need a type of Lipschitz truncation that is able to distinguish between positive and negative part of a function.

First, we define what it means to be asymmetrically Lipschitz continuous.
\begin{definition}
   Let $\lambda,\mu>0$. For $ t \in \R$ we define 
    \[
    d_{\lambda,\mu}(t) = \begin{cases}
        \lambda t & \text{if } t>0,  \\
        - \mu t  & \text{if } t>0.
    \end{cases}
    \]
    and write $d_{\lambda,\mu}(x,y) = d_{\lambda,\mu}(x-y)$. For $x,y \in \R^n$ we define
    \[
    d_{\lambda,\mu}(x,y) = \sum_{i=1}^n d_{\lambda,\mu}(x_i-y_i).
    \]
\end{definition}
\begin{remark}
    Observe that $d_{\lambda,\mu}$ is not a metric, as it is non-symmetric. It is however positive definite and satisfies the triangle inequality. This can be shown by a simple case distinction.
\end{remark}
\begin{definition}
    Let $F \colon X \subset \R^n \to \R$. We say that $F$ is (asymetrically) $(\lambda,\mu)$-Lipschitz continuous if for any $x,y \in X$ 
    \begin{equation} \label{def:asLip}
        u(x) - u(y) \leq d_{\lambda,\mu}(x,y).
    \end{equation}
\end{definition}
Observe that this inequality also (by swapping roles of $x$ and $y$) gives the lower bound 
\[
u(x) - u(y) \geq - d_{\lambda,\mu}(y,x).
\]
Therefore, if $\lambda= \mu$, $(\lambda,\mu)$-Lipschitz continuity is equivalent to $\lambda$-Lipschitz continuity with respect to the $\ell^1$-norm on $\R^n$.

The following extension lemma then follows the standard argument by \textsc{Kirszbraun} \cite{Kirszbraun}:

\begin{lemma} \label{lemma:extension}
    Suppose that $X \subset \R^n$ is closed and $F \colon X \to \R$ is asymmetrically $(\lambda,\mu)$-Lipschitz continuous. Then there exists a function $\tilde{F} \colon \R^n \to \R$ that coincides with $F$ on $X$ and is $(\lambda,\mu)$-Lipschitz continuous.
\end{lemma}
\begin{proof}
    It suffices to show the following: If $z \notin X$, then we can extend the function $F \colon X \to \R$ to $\tilde{F} \colon X \cup \{z\} \to \R$ while retaining the property of being $(\lambda,\mu)$-Lipschitz continuous. By an inductive argument one then may show that for any finite set $Z$ there exists an extension $\tilde{F} \colon X \cup Z \to \R$ and, by considering the limit, we can show it for any countable set $Z$. By picking a countable set $Z$ that is dense in $\R^n \setminus X$, we can afterwards extend $\tilde{F} \colon X \cup Z \to \R$ to $\tilde{F} \colon \R^n \to \R$ by extending continuously.

    To show the initial claim (extension to one point), it suffices to show that we can pick an $a \in \R$, such that  $\tilde{F} \colon X \to \R$ is $(\lambda,\mu)$-Lipschitz continuous if we set $\tilde{F}(z)=a$. The set of such $a$ is characterised by all $a$ that satisfy
    \[
        a -u(x) \leq d_{\lambda,\mu}(z,x) \quad \text{and} \quad u(y)-a \leq d_{\lambda}(y,z) \quad \text{for all } x,y \in X.
    \]
    If this set were empty, there would exist $x,y \in X$ such that
    \[
    u(y) -u(x) > d_{\lambda,\mu}(y,z)+d_{\lambda,\mu}(z,x).
    \]
    Due to the previously mentioned triangle inequality, this is a contradiction as it would entail 
    \[
    u(y) -u(x) > d_{\lambda,\mu}(y,x).
    \]
\end{proof}

The other ingredient to formulating an asymmetric version of Lipschitz truncation for a function $u \colon \R^n \to \R$ is finding an appropriate set $X$ in which $u$ is $(\lambda,\mu)$-Lipschitz continuous.

For a function $v$ we write 
\[
(v)_+ = \max(v,0), \quad (v)_- = \max(-v,0) \quad \text{such that} \quad v = (v)_+-(v)_-.
\]
For vector valued function we use this notation component-wise, i.e.
\[
((v)_+)_i = (v_i)_+.
\]
As the coordinate directions $(e_i)$ and $(-e_j)$ are not treated equally by our definition of Lipschitz continuity (even for $i=j$), we also need another version of the maximal function. A suitable modification is provided by the following lemma:

\begin{lemma} \label{lemma:MF}
    For $u \in L^1_{loc}(\R^n)$ we define 
    \[
        \Ncal u(x) \coloneqq \sup_{r_1,\ldots,r_n} \frac{1}{2^n r_1 \ldots r_n} \int_{Q_{r_1,\ldots r_n(x)}} \vert u (z) \vert \dz,
    \]
    where for $x=(x_1,\ldots,x_n) \in \R^n$, $Q_{r_1,\ldots,r_n}(x) = (x_1-r_1,x_1+r_1) \times \ldots (x_n-r_n,x_n+r_n)$.
    Then \begin{enumerate} [label=(\roman*)]
        \item $\Ncal$ is sublinear, i.e. $\Ncal(u+v) \leq \Ncal(u)+\Ncal(v)$;
        \item $\Ncal$ is bounded from $L^p$ to $L^p$ for any $1<p \leq \infty$ and in particular there is a dimensional constant $c>0$ such that
        \[
          \Vert \Ncal u \Vert_{L^p} \leq  c \left( \tfrac{p}{p-1} \right)^n \Vert u \Vert_{L^p};
        \]
        \item For each $0< \varepsilon$
        \[
            \LL^n( \{ \Ncal u \geq \lambda\}) \leq \frac{C(\varepsilon)}{\lambda^{1+\varepsilon}} \int_{\{ \vert u \vert \geq \lambda/2\}} \vert u \vert \dx.
        \]
    \end{enumerate}  
\end{lemma}
Two facts about the operator $\Ncal$ should be highlighted: First of all, the statement is only valid \emph{without} any rotation-invariance. If one would add 
\[
\tilde{N} u = \sup_{r_1,\ldots,r_n} \sup_{R \in \mathrm{SO}(n)} \frac{1}{2^n r_1 \cdot \ldots \cdot r_n} \int_{RQ_{r_1,...,r_n}(0)} \vert u(x+z) \vert \dz,
 \]
 then this operator is not bounded for general $p$. And secondly, the operator is \emph{not} bounded from $L^1$ to $L^{1,\infty}$; and if we consider Orlicz-type spaces, we lose multiple logarithms, i.e. $\Ncal \colon L^{\psi} \to L^1$ is bounded, where $\psi (t) = t (\log(2+t))^n$.
\begin{proof}
The first assertion is clear by definition. For the second assertion, boundedness from $L^{\infty} \to L^{\infty}$ (with constant $1$) is also simple. For $1<p<\infty$, we may reduce to continuous functions: Note that due to sublinearity, if $(u_j)_{j \in \N}$ is a Cauchy-sequence of continuous functions in $L^p$ then
\[
\vert \Ncal (u_j) - \Ncal(u_j) \vert \leq \Ncal(u_j-u_k).
\]
Therefore, $\Ncal(u_j)$ converges (to $\Ncal u$, where $u$ is the limit of $u_j$) if $\Ncal$ is bounded from $L^p$ to $L^p$ for continuous functions.

Now define 
\[
\M_i u = \sup_{r >0} \frac{1}{2r} \int_{-r}^r u(x+z e_i) \dz.
\]
If $u \in C_b(\R^n)$, this operation is well-defined and $\M_i u$ is a continuous function. Moroever, due to boundedness of the maximal function (and Fubini's theorem) in one dimension
\[
\Vert \M_i u \Vert_{L^p} \leq C(p) \Vert u \Vert_{L^p},
\]
where $C(p) \leq Cp/(p-1)$ is the constant asscociated to the maximal function mapping $L^p$ to $L^p$ in one dimension.
Observe that
\[
\tilde{\M} = \M_n \circ \M_{n-1} \circ \ldots \M_1
\]
is still bounded from $L^p$ to $L^o$ and 
\[
\Vert \tilde{M} u \Vert_{L^p} \leq C(p)^n \Vert u \Vert_{L^p}.
\]
But $\Ncal \leq \tilde{M}$ almost everywhere, as for any $r_1,\ldots,r_n$ we have
\begin{align*}
    \frac{1}{2^n r_1\ldots r_n} \int_{Q_{r_1,\ldots,r_n}} \vert u (z) \vert \dz 
    &\leq \frac{1}{2^{n-1}r_2 \ldots r_n} \int_{Q_{0,r_2,\ldots,r_n}(x)} \vert \M_1 u(z) \vert \dH^{n-1}(z) \\
    & \leq \ldots \\
    & \leq \M_n \circ \ldots \M_1 u(x)
\end{align*}
Thus, for all $u \in C_b(\R^n) \cap L^p(\R^n)$ we have
\[
\Vert \Ncal u \Vert_{L^p} \leq C(p)^n \Vert u \Vert_{L^p}
\]
As $C(p) \leq  \frac{Cp}{p-1}$, the second assertion is proven.

The last statement is a rather standard analogue of the statement for the maximal function $\M$. As $u \colon L^p \to L^p$ is bounded it is also bounded from $L^p$ to the weak space $L^{p,\infty}$. Thus, for any function $v$ we have
\[
\LL^n(\{ \Ncal v \geq \lambda\} ) \leq \frac{C(\varepsilon)}{\lambda^{(1+\varepsilon)}} \int \vert v \vert^{1+\varepsilon} \dz.
\]
We apply this inequality to the function
\[
    v = u_{\lambda} = \begin{cases}
        \vert u \vert -\tfrac{\lambda}{2} & \text{if } \vert u \vert > \tfrac{\lambda}{2}, \\
        0 & \text{else.}
    \end{cases}
\]
Due to sublinearity, we have $\{\Ncal u \geq \lambda\} \subset \{\Ncal u_{\lambda} \geq \lambda/2\}$ and therefore
\begin{align*}
\LL^n( \{ \Ncal u \geq \lambda \}) &\leq \LL^n( \Ncal u_{\lambda} \geq \tfrac{\lambda}{2}) \\
&\leq \frac{C(\varepsilon)}{\lambda^{(1+\varepsilon)}} \int_{\{\vert u \vert \geq \lambda/2\}} \vert u \vert^{(1+\varepsilon)} \dx.
\end{align*}
\end{proof}
In Lemma \ref{lemma:extension} we have shown by elementary means that the concept of extension of Lipschitz functions also applies to the present definition of asymmetric Lipschitz continuity. In \cite{AF84,Zhang} it is shown that $W^{1,1}$- functions are Lipschitz continuous on sublevel sets of the maximal function. The following lemma provides the counterpart to this statement for asymmetric Lipschitz continuity. As not all coordinate directions can be treated equally, we need the modified maximal function $\Ncal$:
\begin{lemma} \label{lemma:AF}
    Suppose that $u \in W^{1,1}_{loc}(\R^n;\R)$. There exists a dimensional constant $C=C(n)$ such that for all $\lambda,\mu >0$ and 
    \[
    x,y \in \{ \Ncal (\nabla u)_+ \leq \lambda\} \cap \{ \Ncal(\nabla u)_- \leq \mu \} \eqqcolon X
    \]
    is asymmetrically $(C \lambda, C \mu)$-Lipschitz continuous on $X$.
\end{lemma}
\begin{proof}
    Observe that the set $X$ is a closed set and suppose that $x,y$ are Lebesgue points of the function $u$. We further suppose that $x_i \neq y_i$ for any index $i$ and define $r_i \coloneqq \vert x_i -y_i \vert$ (In the case where $x_i=y_i$ for some $i$, the proof proceeds in the same fashion if we choose $r_i = \delta$ for an arbitrarily small $\delta >0$. The observation \eqref{psiops} below then may not be valid, but $(x_i-y_i)=0$ and both other terms are rather small).

    We define $Q_1 = Q_{r_1,\ldots,r_n}(x)$, $Q_2 = Q_{r_1,\ldots,r_n}(y)$ and $Q = Q_1 \cap Q_2$. Observe that for any $z \in Q$ we have that
    \begin{equation} \label{psiops}
    (x_i-z_i), \quad (z_i-y_i), \quad (x_i-y_i)
    \end{equation}
    have the same sign. We further write $(u)_Q = \fint_Q u$.
    We aim to show
    \begin{enumerate} [label=(\roman*)]
        \item \label{lemma:AF:item:1} $(u)_Q - u(y) \leq C d _{\lambda,\mu}(x,y)$;
        \item $ u(x) -(u)_Q \leq C d_{\lambda,\mu}(x,y)$.
    \end{enumerate}
    We only show the first assertion, the proof for the second is quite similar.
    We use the Fundamental Theorem of Calculus and obtain
    \begin{align*}
        (u)_Q -u(y) &= \fint_Q u(z) -u(y) \dz \\
        &= \fint_Q \int_0^1 \nabla u(y+t(z-y)) \cdot (z-y) \dt \dz \\
        &= \tfrac{1}{\LL^n(Q)} \int_0^1 \tfrac{1}{t^n} \int_{t Q(0)} \nabla u(y+tz) \cdot  \tfrac{z}{t} \dz \dt = (\ast)
    \end{align*}
    Here, $tQ(0)$ denotes the cube $Q$ translated by $-y$ and rescaled by $t$. Further, recall that $\tfrac{z_i}{t}$ has the same sign as $(x-y)_i$. Therefore,
    \begin{align*}
        (\ast) & \leq \int_0^1 \fint_{tQ(0)} \nabla u(x+z) \cdot (x_i-y_i) \dx \dt 
     \\
     & \leq \int_0^1 \sum_{i=1}^n \begin{cases}
         \Ncal (\partial_i u)_+ (x_i-y_i) & \text{if } y_i-x_i \geq 0 \\
         -\Ncal (\partial_i u)_- (x_i-y_i) & \text{if } y_i-x_i < 0
     \end{cases} \dt
     \\
     & \leq \int_0^1 C d_{\lambda,\mu}(x,y) \dt = C d_{\lambda,\mu}(x,y).
    \end{align*}
    By triangle inequality we then can show that
    \[
    u(x)-u(y) \leq C d_{\lambda,\mu}(x,y)
    \]
    and therefore $u$ is asymmetrically Lipschitz continuous on the set $X$.
\end{proof}

As a simple consequence we then obtain the following truncation statement:
\begin{coro}
    Suppose that $u \in W^{1+\varepsilon}(\R^n;\R)$. For any $\lambda,\mu>0$ there exists a function $u_{\lambda,\mu} \in W^{1,\infty}(\R^n;\R)$ such that 
    \begin{enumerate} [label=(T\arabic*)]
        \item \label{T1} $\Vert (\nabla u_{\lambda,\mu})_+ \Vert_{L^{\infty}} \leq C \lambda$;
        \item \label{T2} $\Vert (\nabla u_{\lambda,\mu})_- \Vert_{L^{\infty}} \leq C \mu$;
        \item \label{T3} The set $\{\nabla u \neq \nabla u_{\lambda,\mu} \} \subset \{ u \neq u_{\lambda,\mu} \}$ is contained in the set $\{ \Ncal (\nabla u)_+ \geq \lambda\} \cup \{ \Ncal (\nabla u)_- \geq \mu\}$.
        \item \label{T4} Consequently, there is a dimensional constant $C(\varepsilon)= (C(n) \varepsilon^{-n} (1+\varepsilon)^n)^{1+\varepsilon}$ such that 
        \[
        \LL^n(\{ u \neq u_{\lambda,\mu}\}) \leq \frac{C(\varepsilon)}{\lambda^{1+\varepsilon}} \int_{\{ \vert (\nabla u)_+ \vert \geq \lambda/2 \}} \vert (\nabla u)_+ \vert^{1+\varepsilon} \dx  + \frac{C(\varepsilon)}{\mu^{1+\varepsilon}} \int_{\{ \vert (\nabla u)_- \vert \geq \mu/2 \}} \vert (\nabla u)_- \vert^{1+\varepsilon}\dx.
        \]
    \end{enumerate}
\end{coro}
\begin{proof}[Proof(Sketch):]
This follows by combining the previous Lemmas \ref{lemma:extension}, \ref{lemma:MF} and \ref{lemma:AF}. In particular, we leave $u$ untouched on the set $X = \{ \Ncal (\nabla u)_+ \leq \lambda \} \cap \{\Ncal (\nabla u)_- \leq \mu\}$. Due to Lemma \ref{lemma:AF} the function $u$ is $(C\lambda,C\mu)$-Lipschitz continuous on this set. We then can extend the function to the entirety of $\R^n$ by Lemma \ref{lemma:extension}. This shows \ref{T1}--\ref{T3}. The estimate of \ref{T4} then is shown by the last item of Lemma \ref{lemma:MF}.
\end{proof}

For the proof of Theorem \ref{thm:main3} we also need a version that preserves zero boundary values for a domain. For usual Lipschitz truncation of a function $u \in W^{1,1}_0(\Omega)$ a version of this statement can, for instance, be found in \cite{FJM}. The assumption there is that the boundary of the set $\Omega$ is Lipschitz (the existence of an exterior cone at any $x \in \Omega$ is enough). At least with the technique of proof displayed in Lemma \ref{lemma:AF} such a general assumption is seems impossible. We therefore additionally assume that $\Omega$ is convex (which, in a way, is an 'exterior cone' condition with the cone being a half-space).

\begin{lemma} \label{lemma:AF2}
     Suppose that $\Omega$ is a convex domain and $u \in W^{1,1}(\Omega;\R)$. There exists a dimensional constant $C=C(n)$ such that for all $\lambda,\mu >0$ and 
    \[
    x,y \in \{ \Ncal (\nabla u)_+ \leq \lambda\} \cap \{ \Ncal(\nabla u)_- \leq \mu \} \eqqcolon X
    \]
    is asymmetrically $(C \lambda, C \mu)$-Lipschitz continuous on $Y:=(X \cap \Omega) \cup \partial{\Omega}$.
\end{lemma}
\begin{proof}[Proof (Sketch):]
Suppose that $x,y \in Y$. Our goal is to show property \eqref{def:asLip}. In view of Lemma \ref{lemma:AF}, if both $x,y \in (X \cap \Omega)$, the statement is proven. If $x,y \in \partial \Omega$, then the statement is immediate as $u(x)=u(y)=0$. So the only case of interest is when one of the points is in $X \cap \Omega$ and the other is in $\partial \Omega$.

As the proofs are very similar, let us assume that $x \in \partial \Omega$ and $y \in X \cap \Omega$. Let $\bar{x} =x +(x-y)$ and $r_i = \vert x_i -y_i \vert$. Due to convexity of the domain one may show that one  cube $Q$ of the form $Q = \bar{x} + I_1 \times \ldots \times I_n$ where $I_1 \in \{(0,r_i/2),(-r_1/2,0)\}$ (i.e. is either the positive or negative interval) is entirely in $\Omega^C$. As we may extend $u$ by $0$ to the entirety of $\Omega$, $(u)_Q =0 =u(\bar{x}) =u(x)$. We can now proceed similarly to item \ref{lemma:AF:item:1} in the proof of Lemma \ref{lemma:AF}.
\end{proof}

As a direct consequence, as before we also obtain.

\begin{coro} \label{coro:truncation:boundary}
    Suppose that $\Omega \subset \R^n$ is a convex domain. There is a dimensional constant $C =C(n)$ such that for any $u \in W_0^{1+\varepsilon}(\R^n;\R)$ and any $\lambda,\mu>0$ there exists a function $u_{\lambda,\mu} \in W_0^{1,\infty}(\Omega;\R)$ such that 
    \begin{enumerate} [label=(T\arabic*)]
        \item \label{T1} $\Vert (\nabla u_{\lambda,\mu})_+ \Vert_{L^{\infty}} \leq C \lambda$;
        \item \label{T2} $\Vert (\nabla u_{\lambda,\mu})_- \Vert_{L^{\infty}} \leq C \mu$;
        \item \label{T3} The set $\{\nabla u \neq \nabla u_{\lambda,\mu} \} \subset \{ u \neq u_{\lambda,\mu} \}$ is contained in the set $\{ \Ncal (\nabla u)_+ \geq \lambda\} \cup \{ \Ncal (\nabla u)_- \geq \mu\}$.
        \item \label{T4} Consequently, there is a dimensional constant $C=C(n,\varepsilon)$ such that 
        \[
        \LL^n(\{ u \neq u_{\lambda,\mu}\}) \leq \frac{C}{\lambda^{1+\varepsilon}} \int_{\{ \vert (\nabla u)_+ \vert \geq \lambda/2} \vert (\nabla u)_+ \vert^{1+\varepsilon} \dx  + \frac{C}{\mu^{1+\varepsilon}} \int_{\{ \vert (\nabla u)_- \vert \geq \mu/2} \vert (\nabla u)_- \vert^{1+\varepsilon}\dx.
        \]
    \end{enumerate}
\end{coro}
\section{Higher integrability of the determinant} \label{sec:determinant}
This section is devoted to showing both Theorems \ref{thm:main1} and \ref{thm:main2}. As mentioned before we assume that 
\begin{itemize}
    \item $F$ is a quasiconcave function with $p$-growth, that is 
    \[
    \vert F(v) \vert \leq C(1+ \vert v \vert^p);
    \]
    \item $F(0) =0$;
    \item $u \in L^1(\R^n;\R^n)$ has compact support.    
\end{itemize}
Although the statements of Theorems \ref{thm:main1} and \ref{thm:main2} are almost identical (up to the sign of the exponent $\alpha$ of the logarithm), their proof differs quite a bit.
\begin{proof}[Proof of Theorem \ref{thm:main1}] For $\lambda >0$ let $u_{\lambda}$ be the function obtained by Theorem \ref{thm:LT}, i.e. $\Vert \nabla u_{\lambda} \Vert_{L^{\infty}} \leq c'(n) \lambda$ and $\{ u \neq u_{\lambda}\} \subset \{ \M(\nabla u) > \lambda\}$. Observe that for all $\lambda>1$,  $u_{\lambda}$ are uniformly compactly supported.

We demonstrate the proof first whenever $\alpha \geq 0$-the case $\alpha <0$ needs a slight adaptation that is possible, but superfluous for $\alpha \geq 0$. Observe that then $\nabla u \in L^p$ and we can directly use quasiconcavity of $F$. We therefore have
\begin{equation} \label{quasiconcavity:thm1}
\begin{split}
0 &\geq \int_{\R^n} F(\nabla u - \nabla u_{\lambda}) \dx \\
 & = \int_{\{ \M \nabla u > \lambda\}}F(\nabla u) - (F(\nabla u) - F(\nabla u -\nabla u_{\lambda})) \dx = (\ast).
 \end{split}
\end{equation}
As $F$ is quasiconcave (i.e. $-F$ is quasiconvex) it is locally Lipschitz continuous and therefore
\begin{align*}
    (\ast) \geq \int_{\{ \M \nabla u > \lambda\}} F(\nabla u) - C \lambda (\lambda^{p-1} + \vert \nabla u \vert^{p-1}) \dx.
\end{align*}
Consequently,
\begin{equation} \label{eq:consequently}
\int_{\{ \M \nabla u > \lambda\}} F(\nabla u) \dx \leq C \lambda \int_{\{ \M \nabla u > \lambda\}} \vert \nabla u \vert^{p-1} + \lambda^{p-1} \dx. 
\end{equation}
We integrate this inequality in $\lambda$ with an additional weigth of $(1+\lambda)^{-1} \log(1+\lambda)^{\alpha}$ and split the integral on the left-hand-side into its positive and its negative part.
This yields:
\begin{equation} \label{eq:main1:proof:1}
\begin{split}
\int_1^{\infty} &(1+\lambda^{-1}) \log(1+\lambda)^{\alpha}   \int_{\{\M \nabla u > \lambda\}} F_+(\nabla u) \dx \dlambda  \\
&\leq  C \int_1^{\infty} \log(1+\lambda)^{\alpha} \int_{\{ \M \nabla u > \lambda\}} \vert \nabla u \vert^{p-1} + (1+\lambda)^{p-1} \dx \dlambda
\\
& +\int_1^{\infty} (1+\lambda^{-1}) \log(1+\lambda)^{\alpha}   \int_{\{\M \nabla u > \lambda\}} F_-(\nabla u) \dx \dlambda.
\end{split}
\end{equation} 
We now estimate both sides by using Fubini's theorem:  For the left-hand-side we obtain (as $\alpha > -1$)
\begin{align*}
    \int_1^{\infty}& (1+\lambda^{-1}) \log(1+\lambda)^{\alpha}  \int_{\{\M \nabla u > \lambda\}} F_+(\nabla u) \dx \dlambda
    \\
    &= \int_{\{ \M \nabla u> 1\}} F_+(\nabla u) \int_1^{\M \nabla u} (1+\lambda)^{-1} \log(1+\lambda)^{\alpha} \dlambda \dx \\
     &=\int_{\{\M \nabla u \geq 1\} } F_+(\nabla u) \frac{1}{\alpha+1}  (\log(1+\M(\nabla u))^{\alpha +1}-\log(2)^{\alpha+1}) \dx \\
    &\geq  \int_{\{\M \nabla u \geq z_{\alpha} \} } F_+(\nabla u) \frac{1}{2(\alpha+1)} \log(1+\M(\nabla u))^{\alpha +1}   \dx -C,
\end{align*}
where $z_{\alpha}$ is such that $\log(1+z_{\alpha})^{\alpha+1} \geq 2\log(2)^{\alpha+1}$ (so that we can, up to loss of constant, absorb $\log(2)^{\alpha+1}$ into the first term). The constant $C$ then comes from the integral
\[
\int_{\{ 1 < \M \nabla u < z_{\alpha} \}}  F_+(\nabla u) \frac{1}{\alpha+1}  (\log(1+\M(\nabla u))^{\alpha +1}-\log(2)^{\alpha+1}) \dx.
\]
The second term on the right hand side can be estimated from above by splitting the integral into a term where $F_-$ is large and a term where it is controlled
\[
\int_{1}^{\infty} (1+\lambda)^{-1} \log(1+\lambda)^{\alpha} \int_{\{F_-(\nabla u) \geq \lambda^p\}} F_-(\nabla u) \dx \dlambda + \int_{1}^{\infty} (1+\lambda)^{\alpha} \log(1+\lambda)^{-1} \int_{\{\M u >\lambda\}} \lambda^p \dx \dlambda.
\]
The second integral can, with change of constant, be absorbed into the first term on the right-hand-side of \eqref{eq:main1:proof:1}, whereas the second is estimated as before:
\begin{align*}
\int_{1}^{\infty} &(1+\lambda)^{-1} \log(1+\lambda)^{\alpha} \int_{\{F_-(\nabla u)\geq \lambda^p\}} F_-(\nabla u) \dx \dlambda \\
&= \int_{\{F_-(\nabla u) \geq 1\}} F_-(\lambda) \int_1^{F_-(\nabla u)^{1/p}} (1+\lambda)^{-1} (\log(1+F_-(\nabla))^{\alpha}
\\
& \leq \int_{\{F_-(\nabla u) \geq 1\}} \frac{1}{\alpha+1} \log(1+ F_-(\nabla u)^{1/p})^{\alpha + 1} \dx \\
& \leq \int_{F_-(\nabla u) \geq 1} \frac{1}{\alpha+1} \log(1+ F_-(\nabla u)) \dx.
\end{align*}
The first term on the right-hand-side of \eqref{eq:main1:proof:1} can instead be bounded by the following calculation:
\begin{align*}
    C& \int_1^{\infty} \log(1+\lambda)^{\alpha}  \int_{\{ \M \nabla u > \lambda\}} \vert \nabla u \vert^{p-1} + (1+\lambda)^{p-1} \dx \dlambda \\
& = C \int_{\{\M \nabla u > 1\}} \vert \nabla u \vert^{p-1} \int_1^{\M \nabla u} (\log(1+\lambda))^{\alpha} \dlambda + \int_1^{\M \nabla u} (\log(1+\lambda)^{\alpha} (1+\lambda)^{p-1} \dlambda \dx \\
& \leq (\ast)
\end{align*}

\noindent We can now estimate both integral using an upper bound for the primitive of the function. First, observe that 
\begin{equation} \label{estimatex}
\frac{d}{d\lambda} \left((1+\lambda) \log(1+\lambda))^{\alpha}\right) = \log(1+\lambda)^{\alpha} + \alpha \log(1+\lambda)^{\alpha-1},
\end{equation}
and that
\[
\frac{d}{d\lambda} \left((1+\lambda)^p \log(1+\lambda))^{\alpha}\right) =  p(1+\lambda)^{p-1}\log(1+\lambda)^{\alpha} + \alpha (1+\lambda)^{p-1}\log(1+\lambda)^{\alpha-1}.
\]
All summands in the corresponding sums are positive and we can therefore estimate 
\begin{align*}
(\ast) & \leq   C \int_{\{\M \nabla u > 1\}} \vert \nabla u \vert^{p} (\log(1+\M \nabla u))^{\alpha}  \dx \\
& \leq  C \int_{\{\M \nabla u > 1\}} \vert (\M \nabla u) \vert^{p} (\log(1+\M \nabla u))^{\alpha}  \dx.
\end{align*}
As the maximal function is bounded in spaces of type $L^p (\log L)^{\alpha}$ (cf. \cite{Hiro}), there is a constant $C(p,\alpha)$ such that
\begin{align*}
    (\ast) \leq  C(p,\alpha) \int_{\{\M \nabla u > 1\}} \vert (\nabla u) \vert^{p} (\log(1+\nabla u))^{\alpha}  \dx.
\end{align*}
Comparing now the left- and right-hand-side of equation of \eqref{eq:main1:proof:1} yields
\begin{align*}
\int_{\{\M \nabla u > z_{\alpha} \} } &F_+(\nabla u) \frac{1}{\alpha+1}  \log(1+(\nabla u))^{\alpha +1} \dx\\ &\leq C(p,\alpha) \int_{\{\M \nabla u > 1\}} \vert (\nabla u) \vert^{p} (\log(1+\nabla u))^{\alpha}  \dx  \\ 
&+
C(\alpha) \int_{\{F_-(\nabla u) > 1 \} } F_-(\nabla u) \frac{1}{\alpha+1}  \log(1+F_-(\nabla u))^{\alpha +1} \dx
\end{align*}
Observe that on the complement of $\{\M \nabla u >z_{\alpha}\}$ we know that $\vert \nabla u \vert \leq z_{\alpha}$, consequently $F_+(\nabla u) \in L^{\infty}$ on the complement of $\{\M \nabla u >z_{\alpha}\}$.

As $F$ has $p$-growth such that $\log(1+F(\nabla u))^{\alpha+1} \leq (p (c+\log(1+ \vert \nabla u \vert))$ we conclude \[
F(\nabla u) \cdot \log(1+ \vert \nabla u \vert)^{\alpha +1} \in L^1.
\]
\smallskip

For $-1 \leq \alpha < 0$ we need to adapt the proof as $(\nabla u) \notin L^p$) we cannot directly infer \eqref{quasiconcavity:thm1} (and for $\alpha = -1$ the proof afterwards also breaks down). Instead we consider two truncated versions of the functions, namely at level $\lambda$ ($u_{\lambda}$) and level $\lambda^2$ ($u_{\lambda^2}$). Observe that we can use quasiconcavity for the difference of $u_{\lambda}$ and $u_{\lambda^2}$:

\begin{equation} \label{quasiconcavity:2:thm1}
\begin{split}
0 &\geq \int_{\R^n} F(\nabla u_{\lambda^2} - \nabla u_{\lambda}) \dx \\
 & = \int_{\{ \M \nabla u > \lambda\}}F(\nabla u_{\lambda^2}) - (F(\nabla u_{\lambda^2}) - F(\nabla u_{\lambda^2} -\nabla u_{\lambda})) \dx = (\ast)
 \end{split}
\end{equation}
Observe that on the set $\{ \M(\nabla u) < \lambda^2\}$ we have $\nabla u_{\lambda^2} = \nabla u$ and we can therefore write (with the estimates as before) for $\lambda \geq 1$
\begin{align*}
    \int_{\{ \lambda< \M \nabla u \leq \lambda^2\}} F( \nabla u) \dx &\leq \int_{\{\lambda <\M (\nabla u) \leq \lambda^2\}} F(\nabla u) -F(\nabla u-\nabla u_{\lambda}) \dx + \int_{\{\M \nabla u > \lambda^2\}} F(\nabla u_{\lambda^2}-\nabla u_{\lambda}) \dx \\
    &\leq  C\int_{\{\M (\nabla u) >\lambda\}} \lambda^p + \lambda \vert \nabla u \vert^{p-1} \dx + \int_{\{\M \nabla u \}\geq \lambda^2\}}  \lambda^{2p}  \dx
\end{align*}
Again split the term on the left-hand-side into $F_+$ and $F_-$ and estimate
\begin{align*}
   \int_{\{ \lambda< \M \nabla u \leq \lambda^2\}} F_-( \nabla u) \dx  \leq \int_{\{C \lambda^p \leq F_-(\nabla u) \leq  C\lambda^{2p}\}} F_-(\nabla u) \dx + C\int_{\{\M \nabla u > \lambda\}} \lambda^p \dx;
\end{align*}
observe that the upper bound in the first integral is due to the estimate $F_{\pm}(\nabla u) \leq C(1+ \vert \nabla u \vert^p)$ combined with $\vert \nabla u \vert \leq \M (\nabla u)$.

Integration in $\lambda$ with the weight $(1+\lambda)^{-1}\log(1+\lambda)^{-\alpha}$ gives
\begin{equation} \label{eq:sebastianvettel}
\begin{split}
    \int_{1}^{\infty}& (1+\lambda)^{-1} \log(1+\lambda)^{\alpha} \int_{\{ \lambda < \M (\nabla u) \leq \lambda^2\}} F_+(\nabla u) \dx \dlambda \\
    & \leq 
    \int_{1}^{\infty} (1+\lambda)^{-1} \log(1+\lambda)^{\alpha} \int_{\{ C\lambda^p < F_-(\nabla u) \leq C\lambda^{2p}\}} F_+(\nabla u) \dx \dlambda  \\
    &
    +C \int_{1}^{\infty} (1+\lambda)^{-1}  \log(1+\lambda)^{\alpha} \int_{\{\M \nabla u \geq \lambda\}} \lambda^{p} +\lambda (\nabla u)^{p-1} \dx \dlambda\\
    &+C \int_{1}^{\infty} (1+\lambda)^{-1}  \log(1+\lambda)^{\alpha} \int_{\{\M \nabla u \geq \lambda^2\}} \lambda^{2p} \dx \dlambda.
\end{split}
\end{equation}
The estimates of the single terms of \eqref{eq:sebastianvettel} now work in a similar fashion as before by using Fubini's theorem. Observe that in particular:
\begin{enumerate} [label=(\alph*)]
    \item \label{item:a} If $\alpha >-1$, then the arising inner integral (for the first term on the left-and-right-hand side) can be calculated as
    \begin{align*}
        \int_{(\M (\nabla u))^{1/2}}^{\M (\nabla u)} (1+\lambda)^{-1} \log(1+\lambda)^{\alpha} \dlambda &= \frac{1}{1+\alpha}( \log(1+\M(\nabla u)^{\alpha} - \log(1+ (\M(\nabla u)^{1/2})^{\alpha}
    \end{align*}
    which can be estimated from above by $\frac{1}{1+\alpha}\log(1+\M(\nabla u))^{\alpha}$ and from below by $C_{\alpha}\log(1+\M(\nabla u))^{\alpha}$ as for any $\lambda$ large enough $\log(1+\lambda^{1/2}) \leq 1/2 \log(1+2\lambda^{1/2}+\lambda) \leq \log(3) + \log(1+\lambda)$.
    \item \label{item:b} If $\alpha =1$ the integral in \ref{item:a} equals
     \begin{align*}
        \int_{(\M (\nabla u))^{1/2}}^{\M (\nabla u)} (1+\lambda)^{-1} \log(1+\lambda)^{\alpha} \dlambda &=  \log(\log(1+\M(\nabla u)))- \log(\log(1+ (\M(\nabla u)^{1/2}))).
    \end{align*}
    Observe that (for $\lambda$ large enough) this difference is bounded from above and below by a constant: It is bounded from above by 
    \[
    \log(\log(\M(\nabla u)))- \log(\log((\M (\nabla u)^{1/2})))  = -\log(1/2) = \log(2)
    \]
    and for $\M(\nabla u) \geq 4$ from below by
    \[
    \log(\log((1+\M(\nabla u))^{1/2})^{3/2})) - \log(\log(1+\M(\nabla u)^{1/2})) = \log(3/2).
    \]
    \item \label{item:c} For the remaining integrals we again note that for a factor $r>0$ we have for $\lambda >1$
    \[
    \frac{d}{d\lambda} (\lambda^r \log(1+\lambda)^{\alpha}) = \lambda^{r-1} \log(1+\lambda)^{\alpha} + \frac{\alpha \lambda^r}{1+\lambda} \log(1+\lambda)^{\alpha -1} \geq c \lambda^{r-1} \log(1+\lambda)^{\alpha},
    \]
    as the first term dominates the second one as $\lambda \to \infty$.
    and we can use this to estimate the integration of $\lambda^{r-1} \log(1+\lambda)^{\alpha}$.
    \item \label{item:d} Note that for the integration in $\lambda$ of the third term of \eqref{eq:sebastianvettel} we have 
    \begin{align*}
    \int_1^{(\M(\nabla u))^{1/2}} \lambda^{2p-1} \log(1+\lambda)^{-1} \dlambda &\leq C (\M(\nabla u)^{1/2})^{2p} \log(1+(\M(\nabla u)^{1/2})^{\alpha} \\
    &\leq C (\M(\nabla u))^p \log(1+ \M(\nabla u))^{\alpha},
    \end{align*}
    i.e. the estimate for the term will be the same as for the second of \eqref{eq:sebastianvettel}.
\end{enumerate}
Using \ref{item:a}-- \ref{item:d}, Fubini's theorem and the same calculations as in the case $\alpha =0$ we can infer from \eqref{eq:sebastianvettel} for $\alpha >-1$
\begin{align*}
\int_{\{\M \nabla u > 2 \} } &F_+(\nabla u) \frac{1}{\alpha+1}  \log(1+(\nabla u))^{\alpha +1} \dx\\ &\leq C(p,\alpha) \int_{\{\M \nabla u > 1\}} \vert (\nabla u) \vert^{p} (\log(1+\nabla u))^{\alpha}  \dx  \\ 
&+
C(\alpha) \int_{\{F_-(\nabla u) > 1 \} } F_-(\nabla u) \frac{1}{\alpha+1}  \log(1+F_-(\nabla u))^{\alpha +1} \dx
\end{align*}
and for $\alpha =-1$
\begin{align*}
\int_{\{\M \nabla u > 2\} } F_+(\nabla u) \dx &\leq C(p,\alpha) \int_{\{\M \nabla u > 1\}} \vert (\nabla u) \vert^{p} (\log(1+\nabla u))^{\alpha}  \dx  
\\&+
C(\alpha) \int_{\{F_-(\nabla u) > 1 \} } F_-(\nabla u) \dx.
\end{align*}
These estimates yield  \[
F(\nabla u) \cdot \log(1+ \vert \nabla u \vert)^{\alpha +1} \in L^1,
\]
finishing the proof.
\end{proof}

\subsection{The case $\alpha < -1$}
\begin{proof}[Proof of Theorem \ref{thm:main2}]
The difference to the proof of Theorem \ref{thm:main1} lies in two steps: First, we need to take care more care of the maximal function as it may appear with a negative power in some terms. Second, $(1+t) \log(1+t)^{\alpha}$ becomes integrable on $(1,\infty)$ and therefore the overall strategy must be different.
For $\lambda >0$ let $u_{\lambda}$ be the function obtained by Theorem \ref{thm:LT}, i.e. $\Vert \nabla u_{\lambda} \Vert_{L^{\infty}} \leq c'(n) \lambda$ and $\{ u \neq u_{\lambda}\} \subset \{\M (\nabla u) > \lambda\}$. Due to quasiconcavity we have for any $\lambda >0$ 
\begin{align*}
    0  & \geq \int_{\R^n} F(\nabla u_{\lambda}) \dx \\
    & = \int_{\{ \M(\nabla u) \leq \lambda \}} F(\nabla u) \dx + \int_{\{ \M (\nabla u) > \lambda\}} F(\nabla u_{\lambda})) \dx
\end{align*}
We rewrite this as
\begin{align}
    \int_{\{ \M(\nabla u) \leq \lambda \}} F(\nabla u) \dx \leq - \int_{\{ \M (\nabla u) > \lambda\}} F(\nabla u_{\lambda})) \dx
\end{align}
and estimate both sides from below and above, respectively.
Let $\lambda >1$ such that we may estimate $ \vert F(v) \vert  \leq C \lambda^p$ for any $\vert v \vert < \lambda$.

The left-hand-side may then be estimated via 
\begin{align*}
    &\int_{\{ \M(\nabla u) \leq \lambda \}} F(\nabla u) \dx \geq \int_{\{ \M (\nabla u) \leq \lambda \}} F_+(\nabla u) \dx -  \int_{\{ \M (\nabla u) \leq \lambda\}} F_-(\nabla u) \dx \\
    & \quad \geq \int_{\{F_+(\nabla u) \leq C \lambda^p\}} F_+(\nabla u) \dx - \int_{\{F_+(\nabla u) \leq C \lambda^p\} \cap \{ \M \nabla u > \lambda\}} F_+(\nabla u) \dx - \int_{\{ \vert \nabla u \vert \leq \lambda\}} F_-(\nabla u)\dx \\
    & \quad\geq \int_{\{F_+(\nabla u) \leq C \lambda^p\}} F_+(\nabla u) \dx  - C \lambda^p \LL^n(\M (\nabla u) \geq \lambda) -\int_{\{ F_-(\nabla u) \leq C \lambda^p \}} F_-(\nabla u)\dx 
    \\
    & \quad\geq \int_{\{F_+(\nabla u) \leq C \lambda^p\}} F_+(\nabla u) \dx  - C \lambda^{p-1} \int_{\{\vert \nabla u \vert \geq \lambda/2\}} \vert \nabla u \vert \dx  -\int_{\{ F_-(\nabla u) \leq C \lambda^p \}} F_-(\nabla u)\dx 
\end{align*}
For the right-hand-side we may simply estimate.
\begin{align*}
     - \int_{\{ \M (\nabla u) > \lambda\}} F(\nabla u_{\lambda})) \dx & \leq \int_{\{\M (\nabla u) > \lambda\}} F_-(\nabla u_{\lambda}) \dx \\
     & \leq C \lambda^p \LL^n(\{ \M (\nabla u) > \lambda\}) \\
     & \leq C \lambda^{p-1} \int_{\{ \vert \nabla u \vert \geq \lambda/2 \}} \vert \nabla u \vert \dx.
\end{align*}
Combining both bounds yields:
\begin{equation} \label{intermediary}
    \int_{\{F_+(\nabla u) \leq C \lambda^p\}} F_+(\nabla u) \dx \leq C \lambda^{p-1} \int_{ \{ \vert \nabla u \vert \geq \lambda/2\}} \vert \nabla u \vert \dx + \int_{\{ F_-(\nabla u) \leq C \lambda^p \}} F_-(\nabla u) \dx.
\end{equation}
We integrate this formula in $\lambda$ for $\lambda >\lambda_0$ and with an additional weight 
\[
\omega(\lambda) = (\lambda+1)^{-1}  \ln(1+ \lambda)^{\alpha}:
\]
Additionally applying Fubini's theorem to both sides gives for the left-hand-side:
\begin{align*}
   \int_{\lambda_0}^{\infty} \omega(\lambda) &\int_{\{F_+(\nabla u) \leq C \lambda^p\}} F_+(\nabla u) \dx \dlambda  = \int_{\R^n} F_+(\nabla u) \int_{\min\{ (F_+(\nabla u)/C)^{1/p},\lambda_0\}}^{\infty} \omega(\lambda) \dlambda \dx 
   \\
   &\geq \int_{\{F_+(\nabla u) \geq C\lambda_0^p\}} F_+(\nabla u) \int_{(F_+(\nabla u)/C)^{1/p}}^{\infty} \ln(1+\lambda)^{\alpha} (1+\lambda)^{-1} \dlambda \dx \\
   & \geq \frac{-1}{1+\alpha} \int_{\{F_+(\nabla u) \geq C\lambda_0^p\}} F_+(\nabla u) \ln(1+(F_+( \nabla u )/C)^{1/p})^{1+\alpha} \dx \\
   & \geq \frac{C_p}{1+\alpha} \int_{\{F_+(\nabla u) \geq C\lambda_0^p\}} F_+(\nabla u) \ln(1+F_+( \nabla u ))^{\alpha} \dx
\end{align*}
Observe that this estimate works perfectly for $\alpha <-1$, but is not valid for $\alpha \geq -1$.
In contrast, the right-hand-side may be estimated as follows:
\begin{align*}
    \int_{\lambda_0}^{\infty} &\omega(\lambda) \left(C \lambda^{p-1} \int_{\{ \vert \nabla u \vert \geq \lambda/2\}} \vert \nabla u \vert \dx + \int_{\{ F_-(\nabla u) \leq C \lambda^p \}} F_-(\nabla u) \dx \right) \dlambda 
    \\
    & =\int_{\{\vert \nabla u \vert \geq \lambda_0/2\}} \vert \nabla u \vert \int_{\lambda_0}^{2 \vert \nabla u \vert } \lambda^{p-1} \omega(\lambda) \dlambda \dx + \int_{\R^n} F_-(\nabla u) \int_{(F_-(\nabla u)/C)^{1/p}}^{\infty} \omega(\lambda) \dlambda \dx \\
    & \leq \int_{\{\vert \nabla u \vert \geq \lambda_0/2\}} \vert \nabla u \vert \int_{\lambda_0}^{2 \vert \nabla u \vert} (1+\lambda)^{p-2} \ln(1+\lambda)^{\alpha} \dlambda \dx \\ &\quad+ \int_{\R^n} F_-(\nabla u)  \ln(1+(F_-(\nabla u)/C)^{1/p})^{1+\alpha} \dlambda \dx.
\end{align*}
As $p>1$ we can estimate the first integral similarly to the previous proof. The second is estimated as previously (for any $\alpha >0)$, so that we get

\begin{align*}
    \int_{\lambda_0}^{\infty} &\omega(\lambda) \left(C \lambda^{p-1} \int_{\{ \vert \nabla u \vert \geq \lambda/2\}} \vert \nabla u \vert \dx + \int_{\{ F_-(\nabla u) \leq C \lambda^p \}} F_-(\nabla u) \dx \right) \dlambda 
    \\
   & \leq C_p (\int_{\R^n} \vert \nabla u \vert (1+ \vert \nabla u \vert)^{p-1} \ln(1 + \vert \nabla u \vert)^{\alpha} \dx + \int_{\R^n} F_-(\nabla u) \ln(1+F_-(\nabla u))^{1+\alpha} \dx)
\end{align*}
We conclude
\begin{align*}
 & \int_{\{F_+(\nabla u) \geq C\lambda_0^p\}} F_+(\nabla u) \ln(1+F_+( \nabla u ))^{\alpha} \dx \\ &\leq C (\int_{\{\nabla u \neq 0 \}}  (1+ \vert \nabla u \vert)^{p} \ln(1 + \vert \nabla u \vert)^{\alpha} \dx + \int_{\R^n} F_-(\nabla u) \ln(1+F_-(\nabla u))^{1+\alpha} \dx)
\end{align*}
As both integrals on the right-hand-side are finite by assumption, the left-hand-side also is. Therefore, $F_+(\nabla u) \ln(1+F_+(\nabla u))^{\alpha} \in L^1$.
\end{proof}
\section{Higher integrability of solutions to the $p$-Laplacian} \label{sec:pLaplace}

The aim of this section is to prove Theorem \ref{thm:main3} together with Corollary \ref{coro:main4}. We prove the \emph{local} version, i.e. we assume that $\Omega \subset \R^n$ is a convex domain and that $u \in W^{1,r}(\R^n;\R)$, $r<p$ is a solution to the equation
\begin{equation} \label{equation:sec:p:Laplace}
  \begin{cases}   \divergence(A(x,\nabla u(x)) = \divergence f & \text{in } \Omega, \\
    u = 0 & \text{on } \partial \Omega.
\end{cases}
\end{equation}
We remind the reader that, qualitatively, the result is not limited to this case: In the case where $\Omega \subset \R^n$ we usually might use a cut-off function $\varphi$ and consider $\varphi u$ (both in the equation as well as the test function) and argue via Sobolev embedding (provided with some stronger conditions on the exponents $p$ and $q$), cf. also Remark \ref{last:remark}. A similar argumentation (for a local result) then can be employed for solutions on non-convex domains $\Omega$.
Observe that the following result is enough for proving Theorem \ref{thm:main3}:

\begin{lemma} \label{lemma:main3}
    Suppose that $\Omega \subset \R^n$ is a convex domain and that $u \in W^{1,r}_0(\Omega;\R)$ is a solution to the equation
    \begin{equation}
            \divergence(A(x,\nabla u)) = \divergence f.
    \end{equation}
    where $A$ obeys the properties \ref{item:A1} and \ref{item:A2}, $f \in  L^{\infty}(\Omega;\R^n)$ and $\max \{1,p-1\} < r < p$. Suppose further that
    \[
    (\partial_i u)_- \in L^q \quad \text{for all } i=1,\ldots,n, \quad r<q<p.
    \]
    Then there exists an $\varepsilon(p,q)$ such that
    \[
    \nabla u \in L^{p+\varepsilon} .
    \]
\end{lemma}

\begin{proof}
    Let $\lambda >0$ and let $\alpha >0$ be a parameter that is later set to $\alpha =\frac{r+1}{q+1}$. We set $\mu = \lambda^{\alpha}$ and denote by $u_{\lambda} = u_{\lambda,\mu}$ the truncation of the function from Corollary \ref{coro:truncation:boundary}. To reiterate, this means that $u_{\lambda} \in W^{1,\infty}_0(\Omega;\R)$ and that  there is a purely dimensional constant $C>0$ such that:
    \begin{enumerate} [label=(U\arabic*)]
        \item \label{U1} $\Vert (\nabla u_{\lambda})_+ \Vert_{L^{\infty}} \leq C \lambda$;
        \item \label{U2} $\Vert (\nabla u_{\lambda})_- \Vert_{L^{\infty}} \leq C \lambda^{\alpha}$;
        \item \label{U3} The set $\{\nabla u \neq \nabla u_{\lambda} \} \subset \{ u \neq u_{\lambda} \}$ is contained in the set $\{ \Ncal (\nabla u)_+ >\lambda\} \cup \{ \Ncal (\nabla u)_- >\mu\}$.
        \item \label{U4} There is a constant $C=C(n,\varepsilon)$ such that 
        \[
        \LL^n(\{ u \neq u_{\lambda,\mu}\}) \leq \frac{C}{\lambda^{1+\varepsilon}} \int_{\{ \vert (\nabla u)_+ \vert > \lambda/2\}} \vert (\nabla u)_+ \vert^{1+\varepsilon} \dx  + \frac{C}{\lambda^{(1+\varepsilon)\alpha}} \int_{\{ \vert (\nabla u)_- \vert > \lambda^{\alpha}/2\}} \vert (\nabla u)_- \vert^{1+\varepsilon}\dx.
        \]
    \end{enumerate}
    We may now test the equation with the function $u_\lambda$. We obtain
    \begin{align*}
        \int_{\Omega} A(x,\nabla u) \cdot \nabla u_{\lambda} \dx  = \int_{\Omega} f \cdot \nabla u_{\lambda} \dx.
    \end{align*}
    Observe that $u_{\lambda}$ is uniformly bounded in $W^{1,1}_0$ and therefore the right-hand-side of the equation can be estimated by $\sup_{\lambda>1} \Vert u_{\lambda} \Vert_{W^{1,1}} \Vert f \Vert_{L^{\infty}} \eqqcolon C_f$. Further splitting the integral on the left-hand-side into 
    \[
    G_{\lambda} \coloneqq \{\Ncal (\nabla u)_+ \leq \lambda\} \cap \{\Ncal(\nabla u)_- \leq \lambda\} \quad \text{and} \quad B_{\lambda} \coloneqq \Omega \setminus G_{\lambda}.
    \]
    we obtain
    \begin{equation} \label{eq:finalproof:start}
        \int_{G_{\lambda}} A(x,\nabla u) \cdot \nabla u_{\lambda} \dx \leq C_f - \int_{B_{\lambda}} A(x,\nabla u) \cdot \nabla u_{\lambda}.
    \end{equation}
    We later integrate equation \eqref{eq:finalproof:start} with a weight $\omega(\lambda)$ in the variable $\lambda$, but first we estimate both sides of  the equation from below and above, respectively.

    \noindent \textbf{Estimate of the integral in $G_{\lambda}$:} \\
    We use assumption \ref{item:A1} and \ref{item:A2} and the fact that $u = u_{\lambda}$ on the set $G_\lambda$ to estimate
    \begin{align*}
         \int_{G_{\lambda}}  A(x,\nabla u) \cdot \nabla u_{\lambda} \dx &= \int_{G_{\lambda}} a_1(\vert \nabla u\vert) ( a_2(x)\nabla u(x)) \cdot \nabla u(x) \dx \\
         &\geq  \int_{G_{\lambda} \cap \{\vert \nabla u \vert \geq 1\}}  \vert \nabla u \vert^p \dx
    \end{align*}
   Instead of the set $G_{\lambda} = \{\Ncal (\nabla u)_+ \leq \lambda\} \cap \{\Ncal (\nabla u)_- \leq \lambda^{\alpha}\}$ the domain of integration should be $\{ (\nabla u)_+ \leq \beta \lambda\}$ (with a constant $0<\beta<1$ that may be chosen arbitrarily small later). We therefore estimate:
   \begin{align*}
        &\int_{G_{\lambda} \cap \{\vert \nabla u \vert \geq 1\}}  \vert \nabla u \vert^p \dx \\
   & \geq \int_{\{1 \leq \vert (\nabla u)_+ \vert \leq \beta \lambda\}} \vert (\nabla u)_+ \vert^p \dx - \int_{\{(\Ncal (\nabla u)_+ > \lambda\} \cap \{\vert (\nabla u)_+ \vert \leq \lambda\}} \vert (\nabla u)_+ \vert ^p \dx  \\
   &\quad - \int_{\{(\Ncal (\nabla u)_- > \lambda^{\alpha}\} \cap \{\vert (\nabla u)_+ \vert \leq \beta\lambda\}} \vert (\nabla u)_+ \vert ^p \dx 
   \\
   & \geq \int_{\{1 \leq \vert (\nabla u)_+ \vert \leq \beta \lambda\}} \vert (\nabla u)_+ \vert^p \dx - ( \beta \lambda)^p \left( \LL^n( \{\Ncal (\nabla u)_+ > \lambda \}) + \LL^n( \{\Ncal (\nabla u)_- > \lambda^{\alpha} \}) \right).
   \end{align*}

We may now use the upper bound on the set $\{ \Ncal v > \lambda\}$ of Lemma \ref{lemma:MF} to obtain
\begin{align*}
    \int_{G_{\lambda} \cap \{\vert \nabla u \vert \geq 1\}}  &\vert \nabla u \vert^p \dx 
    \geq  \int_{\{1 \leq \vert (\nabla u)_+ \vert \leq \beta \lambda\}} \vert (\nabla u)_+ \vert^p \dx  \\ &- C_{\varepsilon} \beta^p\lambda^p \left(\lambda^{-1-\varepsilon} \int_{\{ \vert (\nabla u)_+ > \lambda/2\}} \vert (\nabla u)_+\vert^{1+\varepsilon} \dx + \lambda^{-(1+\varepsilon)\alpha} \int_{\{\vert (\nabla u)_- \vert > \lambda^{\alpha}/2\}} \vert (\nabla u)_-\vert^{1+\varepsilon} \dx \right).
\end{align*}

 \noindent \textbf{Estimate of the integral in $B_{\lambda}$:} \\
For the integral on this set we use the structure of the term $A(x,\nabla u)$: If both $(\partial_j u)$ and $(\partial_j u_{\lambda})$ have the same sign, the product $(A(x,\nabla u))_j \cdot \partial_j u_{\lambda}$ will be positive (and thus contributes negatively to the term). We therefore omit those terms and obtain
\begin{align*}
    C_f &- \int_{B_\lambda} A(x,\nabla u) \cdot \nabla u_{\lambda} \dx
    \\
    &
     \leq C +  \int_{B_{\lambda}} C\nu \sum_{j=1}^n  (\partial_j u)_-^{p-1} (\partial_j u_{\lambda})_+ +  (\partial_j u)_+^{p-1} (\partial _j u_{\lambda})_- \dx \\
     &
     \leq C + C \int_{B_{\lambda}} \vert (\nabla u)_+ \vert^{p-1} \vert (\nabla u_{\lambda})_- \vert + \vert (\nabla u)_- \vert^{p-1} \vert (\nabla u_{\lambda})_+ \vert \dx  \\
     & 
     \leq C + C \int_{B_{\lambda}} \lambda^{\alpha} \vert (\nabla u)_+ \vert^{p-1} + \lambda \vert (\nabla u)_- \vert^{p-1} \dx.
\end{align*}
Remember that $B_{\lambda} = \{ \Ncal (\nabla u)_+ > \lambda\} \cup \{ \Ncal (\nabla u)_- > \lambda^{\alpha} \}$. We estimate both integrals by splitting the domain of integration into a part, where $\vert (\nabla u)_+ \vert$ (or $\vert (\nabla u)_- \vert$, respectively) is large itself and a part where it is small (but the maximal function $\Ncal$ is large). For the integral in $(\nabla u)_+$ we obtain
\begin{align*}
    \lambda^{\alpha} \int_{B_{\lambda}} \vert (\nabla u)_+ \vert^{p-1} \dx & \leq \lambda^{\alpha} \int_{\{\vert \nabla u \vert > \lambda\}} \vert (\nabla u)_+ \vert^{p-1} \dx + \lambda^{p-1+ \alpha} \LL^n(B_{\lambda}) \\
    & \leq \lambda^{\alpha} \int_{\{\vert (\nabla u)_+ \vert > \lambda\}} \vert (\nabla u)_+ \vert^{p-1} \dx + \lambda^{p-1+\alpha} \left(\LL^n(\{ \Ncal (\partial u)_+ > \lambda\}) + \LL^n(\{ \Ncal (\partial u)_- > \lambda^{\alpha}) \right) \\
    &
    \leq \lambda^{\alpha} \int_{\{\vert (\nabla u)_+ \vert > \lambda\}} \vert (\nabla u)_+ \vert^{p-1} \dx + C_{\varepsilon} \lambda^{p-2-\varepsilon+\alpha} \int_{\{ \vert (\nabla u)_+ \vert > \lambda/2\}} \vert (\nabla u)_+ \vert^{1+\varepsilon} \dx
    \\
    &  +  C_{\varepsilon} \lambda^{p-1-\varepsilon \alpha} \int_{\{ \vert (\nabla u)_- \vert > \lambda^{\alpha}/2\}} \vert (\nabla u)_- \vert^{1+\varepsilon} \dx
\end{align*} 
The integral in $(\nabla u)_-$ is estimated in the same fashion by
\begin{align*}
   \lambda\int_{B_{\lambda}} \vert (\nabla u)_- \vert^{p-1} \dx & \leq \lambda \int_{\{ \vert (\nabla u)_- \vert > \lambda^{\alpha} \}}  \vert (\nabla u )_- \vert^{p-1} \dx \\
  & + C_{\varepsilon} \lambda^{(p-1)\alpha-\varepsilon} \int_{\{ \vert (\nabla u)_+ \vert > \lambda/2\}} \vert (\nabla u)_+ \vert^{1+\varepsilon} \dx + C_{\varepsilon } \lambda^{1+(p-2-\varepsilon)\alpha} \int_{\{ \vert (\nabla u)_- \vert > \lambda^{\alpha}/2\}} \vert (\nabla u)_- \vert^{1+\varepsilon} \dx
\end{align*}
 We may combine those estimates by setting $q_1 = \max\{ p-2+\alpha,(p-1)\alpha\}$, $q_2 = \max \{ (p-1),1+(p-2)\alpha\}$ to obtain
 \begin{align*}
     C_f &- \int_{B_\lambda} A(x,\nabla u) \cdot \nabla u_{\lambda} \dx \leq C + C \lambda^{\alpha} \int_{\vert (\nabla u)_+ > \lambda\}} \vert (\nabla u)_+ \vert^{p-1} \dx +   C \lambda \int_{\vert (\nabla u)_- > \lambda^{\alpha}\}} \vert (\nabla u)_- \vert^{p-1} \dx \\
     &+ C_{\varepsilon} \lambda^{q_1 -\varepsilon} \int_{\{ \vert (\nabla u)_+ \vert >\lambda/2} \vert (\nabla u)_+ \vert^{1+\varepsilon} \dx + C_{\varepsilon} \lambda^{q_2 -\alpha \varepsilon} \int_{\{ \vert (\nabla u)_- \vert > \lambda^{\alpha}/2 \}} \vert (\nabla u)_- \vert^{1+\varepsilon} \dx.
 \end{align*}

 \noindent \textbf{Summary of estimates and integration in $\lambda$:} \\

 We combine the estimates obtained for the integral on the "good set" and on the "bad set and obtain for any $\lambda>1$:
 \begin{align*}
     \int_{\{1 \leq \vert (\nabla u)_+ \vert < \beta \lambda \}} \vert (\nabla u)_ + \vert^p \dx & \leq C_{\varepsilon} \beta ^p \lambda^{p-1-\varepsilon} \int_{\{\vert (\nabla u)_+ \vert >\lambda/2\}} \vert (\nabla u)_ + \vert^{1+\varepsilon} \dx \\
     &+ C_{\varepsilon} (\beta^p \lambda^{p-(1+\varepsilon)\alpha} + \lambda^{q_2-\varepsilon}) \int_{\{\vert(\nabla u)_- \vert> \lambda^{\alpha}/2\}} \vert \nabla u )_- \vert^{1+\varepsilon} \dx \\
& + C \lambda^{\alpha} \int_{\{\vert ( \nabla u)_+ \vert > \lambda \} } \vert (\nabla u)_+ \vert^{p-1} + C \lambda \int_{\{ \vert (\nabla u)_- \vert > \lambda^{\alpha}\}} \vert (\nabla u)_- \vert^{p-1} \dx \\
& + C_{\varepsilon} \lambda^{q_1-\varepsilon}  \int_{\{\vert (\nabla u)_+ \vert >\lambda/2\}} \vert (\nabla u)_ + \vert^{1+\varepsilon} \dx + C \\
&= \mathrm{(I)}+\mathrm{(II)}+\mathrm{(III)}+\mathrm{(IV)} + \mathrm{(V)} +C 
 \end{align*}

 We integrate this inequality with the weight $\omega(\lambda) = \lambda^{s-p-1}$ for $1<\lambda <\infty$ for some $r<s<p$ to be specified later (recall that $r$ and $q$ were such that $(\nabla u)_+ \in L^r$ and $(\nabla u)_- \in L^q$). Then the left hand side equals
 \begin{align*}
     \int_1^{\infty} \lambda^{s-p-1} \int_{\{1 \leq \vert (\nabla u)_+ \vert < \beta \lambda\}} \vert (\nabla u)_+\vert^p \dx \dlambda & \geq \int_{\{ \vert (\nabla u)_+ \vert > 1\}} \vert (\nabla u)_+ \vert^p \int_{\vert (\nabla u)_ +\vert/\beta}^{\infty} \lambda^{s-p-1} \dlambda \dx \\
    & = \frac{\beta^{p-s}}{s-p} \int_{\{ \vert (\nabla u)_+ >1\}} \vert(\nabla u)_+ \vert^s \dx.
 \end{align*}
 We estimate the first and the fifth of the right-hand-side together. The first term of the right hand side can be estimated through
 \begin{align*}
     \int_{1}^{\infty}\mathrm{(I)}  \lambda^{s-p-1} \dlambda & = C_{\varepsilon} \beta^p  \int_{1}^{\infty} \lambda^{s-2-\varepsilon} \int_{\{ \vert (\nabla u)_+> \lambda/2 \}} \vert (\nabla u)_+ \vert^{1+\varepsilon} \dx \dlambda \\
     & \leq C_{\varepsilon} \beta^p \int_{\{ \vert (\nabla u)_+ \geq 1/2\}} \vert (\nabla u)_+ \vert^{1+\varepsilon}  \int_1^{2 \vert (\nabla u)_+ \vert} \lambda^{s-2-\varepsilon} \dlambda \dx \\
     &\leq 2^{s-1-\varepsilon}C \beta^p \int_{\{\vert (\nabla u)_+ \geq 1/2\}} \vert (\nabla u)_+ \vert^{s} \dx \\
     & \leq C_p \beta^p (C + \int_{\{\vert (\nabla u)_+ \geq 1\}} \vert (\nabla u)_+ \vert^{s} \dx)
 \end{align*}
 Observe that this estimates works whenever $\varepsilon$ (i.e. we may choose $\varepsilon = r-1$). Furthermore, observe that if $\beta$ is chosen small enough, that the integral over the set $\{ \vert (\nabla u)_+\vert >1 \}$ can be absorbed into the previously estimated left-hand-side.
 Likewise, as $q_1- \varepsilon \leq \max\{p-2+\alpha-\varepsilon,(p-1)\alpha-\varepsilon\} < p-1-\varepsilon$ for $\alpha <1$, that
 \begin{align*}
      C_{\varepsilon} \lambda^{q_1-\varepsilon}  \int_{\{\vert (\nabla u)_+ \vert \geq \lambda/2\}} \vert (\nabla u)_ + \vert^{1+\varepsilon} \dx  \leq (C(\beta) + C_{\varepsilon,\alpha} \beta^{p}) \lambda^{p-1-\varepsilon}   \int_{\{\vert (\nabla u)_+ \vert \geq \lambda/2\}} \vert (\nabla u)_ + \vert^{1+\varepsilon} \dx 
 \end{align*}
 i.e. term $\mathrm{(V)}$ can also be absorbed in the first, i.e. in the left-hand-side. In this step we have 
 
 We therefore conclude
 \begin{align*}
     \int_{\{ \vert (\nabla u)_+ >1 \}} \vert (\nabla u)_+ \vert^s \dx 
     \leq C(\varepsilon,\beta)\int_1^{\infty} (\mathrm{(II)} + \mathrm{(III)} + \mathrm{(IV)} + C ) \lambda^{s-p-1} \dlambda.
 \end{align*}
 To finally prove the claim it suffices to show, that (given our assumptions) the right-hand-side is finite.
 Obviously we have
 \begin{align*}
     \int_1^{\infty} C \lambda^{s-p-1} \dlambda \leq \frac{C}{p-s} < \infty.
 \end{align*}
 We also may estimate for $q_3 = \max\{q_2,p-(1+\varepsilon)\alpha + \varepsilon\}$
 \begin{equation} \label{harrypotter}
 \begin{split}
     \int_1^{\infty} \lambda^{s-p-1} \mathrm{(II)} \dlambda &\leq  \int_1^{\infty} \lambda^{s-p-1+q_3-\varepsilon} \int_{\{\vert (\nabla u)_- \vert \geq \lambda^{\alpha}/2\}} \vert(\nabla u)_- \vert^{1+\varepsilon} \dx \dlambda \\
     &\leq \int_{\{\vert(\nabla u)_- \vert > 1/2\}} \vert (\nabla u)_- \vert^{1+\varepsilon}  \int_1^{(2 \vert(\nabla u)_- \vert^{1/\alpha}} \lambda^{s-p-1+q_3 -\varepsilon} \dlambda\dx \\
     &
      \leq C \int_{\{\vert (\nabla u)_- \> 1/2\}} \vert (\nabla u)_- \vert^{1+\varepsilon+(q_3+s-p-\varepsilon)/\alpha} \dx 
 \end{split} \end{equation}
 This integral is finite whenever
 \[
 (1+\varepsilon)+(q_3+s-p-\varepsilon)/\alpha \leq q, 
 \]
 i.e. (as $q_3$ is the maximum of three terms) if the following three conditions are met:
 \begin{align*}
    (1+\varepsilon)+((p-1)+s-p-\varepsilon)/\alpha \leq q  &\quad \Longrightarrow \quad (1+\varepsilon) + (s-1-\varepsilon)/\alpha \leq q \\
    (1+\varepsilon)+(1+(p-2)\alpha+s-p-\varepsilon)/\alpha \leq q &\quad \Longrightarrow \quad (p-1+\varepsilon)+(1+s-p-\varepsilon)/\alpha \leq q \\
    (1+\varepsilon)+(p-(1+\varepsilon)\alpha+\varepsilon+s-p-\varepsilon)/\alpha \leq q &\quad \Longrightarrow \quad s/\alpha \leq q.
 \end{align*}
    Observe that these inequalities are true for $s=q$ and $\alpha =1$. In particular, for fixed $\alpha$, the last term gives the strictest bound, i.e. the integral \eqref{harrypotter} is finite iff $ s \leq \alpha q$.

    We now come to the integral of the term $\mathrm{(III)}$:
    \begin{align*}
        \int_{1}^{\infty} \lambda^{s-p-1} \mathrm{(III)}\dlambda & \leq \int_1^{\infty} \lambda^{s-p-1+\alpha} \int_{\{\vert (\nabla u)_+ \vert >\lambda\}} \vert (\nabla u)_+ \vert^{p-1} \dx \dlambda \\ 
        & \leq \int_{\{\vert (\nabla u)_+ \geq 1 \}} \vert (\nabla u)_+ \vert^{p-1} \int_1^{\vert (\nabla u)_+ \vert} \lambda^{s-p-1+\alpha} \dlambda \dx \\
       & \leq C \int_{ \{\vert (\nabla u)_+ \vert \geq 1} \vert (\nabla u)_+ \vert^{s-1+\alpha} \dx.
    \end{align*}
    This integral is finite if $s-1+\alpha \leq r$, i.e. $s \leq r+1-\alpha$.
    The integral of $\mathrm{(IV)}$ may be estimated in a similar fashion:
    \begin{align*}
         \int_{1}^{\infty} \lambda^{s-p-1} \mathrm{(IV)}\dlambda &  \leq \int_1^{\infty} \lambda^{s-p} \int_{\{\vert (\nabla u)_- \vert > \lambda^{\alpha}} \vert (\nabla u)_- \vert^{p-1} \dx \dlambda \\
         & \leq \int_{\{\vert (\nabla u)_- >1 \}} \vert (\nabla u)_- \vert ^{p-1} \int_{1}^{\vert (\nabla u)_- \vert^{1/\alpha}} \lambda^{s-p} \dlambda \dx \\
         & \leq C \int_{\{\vert (\nabla u)_- >1 \vert \}} \vert(\nabla u)_- \vert^{(p-1)+(1+s-p)/\alpha} \dx.
    \end{align*}
    Observe that this integral is finite if 
    \[
    q >(p-1) + (1+s-p)/\alpha.
   \]
   Recall that if we choose $\alpha =1-\delta$ (with $\delta$ sufficiently small), that all the terms on the right-hand-side are finite provided that
   \begin{itemize}
       \item $s \leq q/(1-\delta)$ 
       \item $s \leq r +\delta$.
   \end{itemize}
   Hence, choosing $\delta = \frac{q-r}{q+1}$ and therefore $\alpha = \frac{r+1}{q+1}$, we get that
   \[
   (\nabla u)_+ \in L^s \quad \text{for} \quad s= q \frac{r+1}{q+1} >r.
   \]
   This concludes the proof.
\end{proof}

The proof of Theorem \ref{thm:main3} as well as Corollary \ref{coro:main4} are now simple consequences:
\begin{proof}[Proof of Theorem \ref{thm:main3}:]
    We have seen that for any $r<q<p$ there exists and for any solution to the system with $(\nabla u)_+ \in L^r$ and $(\nabla u)_- \in L^q$ we already have $(\nabla u)_+ \in L^{q(r+1)/(q+1)}$. By iteratively increasing the integrability of $(\nabla u)_+$, we get that $(\nabla u)_+ \in L^s$ for any $s<r$.
\end{proof}
Observe that with the current method of proof we cannot obtain the result for $q=1$ (as the function $\Ncal$ is not (weakly) bounded on $L^1$) and also cannot reach $r=q$. The statement $(\nabla u)_- \in L^r \Longrightarrow (\nabla u)_+ \in L^r$ might also be false, but with careful estimates one can probably reach a statement on the growth of the $L^r$ norm of  $(\nabla u)_+$ as $r$ approaches $q$ (cf. for instance \cite{Sbordone96} for results of this type for other problems).

Corollary \ref{coro:main4} is now a simple consequence of Theorem \ref{thm:main3}.
We end the discussion of Theorem \ref{thm:main3} with a remark on a local result:
\begin{remark} \label{last:remark}
In the proof of Lemma \ref{lemma:main3} we have used that $u$ is compactly supported and indeed has zero boundary values on some convex domain $\Omega$. Local results of this type for solutions $u$ on $\R^n$ can be recovered (for a certain range of exponents) by using cut-offs. For instance, suppose that $p >1$ is large enough and that $\varphi$ is a cut-off. We can then repeat the proof by testing instead with $\varphi u_{\lambda}$ instead of $u$; this gives us estimates on the integrability with respect to $\varphi \LL^n$. Observe that the difference to the previous proof emerges when the derivative falls on $\varphi$, i.e. when we need to estimate
\[
\int A(x,\nabla u) \cdot \nabla \varphi u_{\lambda} \dx.
\]
Observe that we have no control over the sign of $\nabla \varphi$, so arguments over the sign of the involved terms are rather difficult. Instead, in the proof it is for sure sufficient that 
\[
\int A(x,\nabla u) \cdot \nabla \varphi u_{\lambda} \dx< \infty.
\]
Observe that for $\nabla u \in L^r$ we have $A(x,\nabla u) \in L^{r/(p-1)}$ and $u_{\lambda} \in L^{nr/(n-r)}$ (or $u_\lambda \in L^{\infty}$ if $r>n$). As a consequence, if $r>n$, we recover a local result for any $r \geq p-1$, as above term is always bounded in $L^{\infty}$; else we need 
\[
\frac{r}{p-1} + \frac{n-r}{nr} \leq 1 \quad \Longrightarrow r \geq  \frac{pn}{n+1}
\]
to ensure validity of a local version of Theorem \ref{thm:main3}. To summarise, with few adjustments, one is able to show the following: Suppose that $u \in W^{1,r}(\R^n;\R)$ is a solution to
\begin{equation*}
            \divergence(A(x,\nabla u)) = \divergence f
\end{equation*}
for $f \in L^{\infty}_{loc}$. Then if $r$ is larger than $\max\{1,p-1,\frac{pn}{n+1}\}$, and $(\nabla u)_- \in L^s$ for some $r < s \leq p$, then $(\nabla u)_+ \in L^{s-\varepsilon}$ for any $s$.
\end{remark}
\bibliography{biblio.bib}

@article {AF84,
    AUTHOR = {Acerbi, Emilio and Fusco, Nicola},
     TITLE = {Semicontinuity problems in the calculus of variations},
   JOURNAL = {Arch. Rational Mech. Anal.},
  FJOURNAL = {Archive for Rational Mechanics and Analysis},
    VOLUME = {86},
      YEAR = {1984},
    NUMBER = {2},
     PAGES = {125--145},
      ISSN = {0003-9527},
   MRCLASS = {49A50 (49D20)},
  MRNUMBER = {751305},
MRREVIEWER = {Anna\ Salvadori},
       DOI = {10.1007/BF00275731},
       URL = {https://doi.org/10.1007/BF00275731},
}

@article{Astala,
author = {Kari Astala},
title = {{Area distortion of quasiconformal mappings}},
volume = {173},
journal = {Acta Mathematica},
number = {1},
publisher = {Institut Mittag-Leffler},
pages = {37 -- 60},
year = {1994},
doi = {10.1007/BF02392568},
URL = {https://doi.org/10.1007/BF02392568}
}

@article {AFS,
    AUTHOR = {Astala, Kari and Faraco, Daniel and Sz\'ekelyhidi, Jr.,
              L\'aszl\'o},
     TITLE = {Convex integration and the {$L^p$} theory of elliptic
              equations},
   JOURNAL = {Ann. Sc. Norm. Super. Pisa Cl. Sci. (5)},
  FJOURNAL = {Annali della Scuola Normale Superiore di Pisa. Classe di
              Scienze. Serie V},
    VOLUME = {7},
      YEAR = {2008},
    NUMBER = {1},
     PAGES = {1--50},
      ISSN = {0391-173X,2036-2145},
   MRCLASS = {35J25 (30C62 35B65 35D10)},
  MRNUMBER = {2413671},
MRREVIEWER = {Leonid\ V.\ Kovalev},
}

@article {Cianchi,
    AUTHOR = {Cianchi, Andrea},
     TITLE = {Optimal integrability of the {J}acobian of orientation
              preserving maps},
   JOURNAL = {Boll. Unione Mat. Ital. Sez. B Artic. Ric. Mat. (8)},
  FJOURNAL = {Bollettino della Unione Matematica Italiana. Serie VIII.
              Sezione B. Articoli di Ricerca Matematica},
    VOLUME = {2},
      YEAR = {1999},
    NUMBER = {3},
     PAGES = {619--628},
      ISSN = {0392-4041},
   MRCLASS = {58C25 (26B10 42B25 46E30)},
  MRNUMBER = {1719554},
MRREVIEWER = {Stephen\ Montgomery-Smith},
}

@article{CT,
    AUTHOR={Colombo, Maria and Tione, Riccardo},    
    TITLE={Non-classical solutions of the p-{L}aplace equation},
    JOURNAL={J. Eur. Math. Soc.},
    YEAR={2024},
    DOI={DOI 10.4171/JEMS/1462},
    }

@article {Faraco,
    AUTHOR = {Faraco, Daniel},
     TITLE = {Milton's conjecture on the regularity of solutions to
              isotropic equations},
   JOURNAL = {Ann. Inst. H. Poincar\'e{} C Anal. Non Lin\'eaire},
  FJOURNAL = {Annales de l'Institut Henri Poincar\'e{} C. Analyse Non
              Lin\'eaire},
    VOLUME = {20},
      YEAR = {2003},
    NUMBER = {5},
     PAGES = {889--909},
      ISSN = {0294-1449,1873-1430},
   MRCLASS = {35J15 (30C62 30G20 49J45 60B10)},
  MRNUMBER = {1995506},
MRREVIEWER = {Andrei\ B.\ Bogatyr\"ev},
       DOI = {10.1016/S0294-1449(03)00014-3},
       URL = {https://doi.org/10.1016/S0294-1449(03)00014-3},
}

@article {FJM,
    AUTHOR = {Friesecke, Gero and James, Richard D. and M\"uller, Stefan},
     TITLE = {A theorem on geometric rigidity and the derivation of
              nonlinear plate theory from three-dimensional elasticity},
   JOURNAL = {Comm. Pure Appl. Math.},
  FJOURNAL = {Communications on Pure and Applied Mathematics},
    VOLUME = {55},
      YEAR = {2002},
    NUMBER = {11},
     PAGES = {1461--1506},
      ISSN = {0010-3640,1097-0312},
   MRCLASS = {74K20 (49J45 74B20 74G65)},
  MRNUMBER = {1916989},
MRREVIEWER = {Georg\ K.\ Dolzmann},
       DOI = {10.1002/cpa.10048},
       URL = {https://doi.org/10.1002/cpa.10048},
}

@article {Gehring,
    AUTHOR = {Gehring, F. W.},
     TITLE = {The {$L\sp{p}$}-integrability of the partial derivatives of a
              quasiconformal mapping},
   JOURNAL = {Acta Math.},
  FJOURNAL = {Acta Mathematica},
    VOLUME = {130},
      YEAR = {1973},
     PAGES = {265--277},
      ISSN = {0001-5962,1871-2509},
   MRCLASS = {30A60},
  MRNUMBER = {402038},
MRREVIEWER = {P.\ Caraman},
       DOI = {10.1007/BF02392268},
       URL = {https://doi.org/10.1007/BF02392268},
}

@article{GR66,
title={Area distortion under quasiconformal mappings}, 
url={https://afm.journal.fi/article/view/134064}, 
DOI={10.5186/aasfm.1966.388}, number={388}, 
journal={Annales Fennici Mathematici}, 
author={Gehring, FW. and Reich, E.}, 
year={1966},
month={Jan.} }

@article {Greco1,
    AUTHOR = {Greco, Luigi and Iwaniec, Tadeusz and Moscariello, Gioconda},
     TITLE = {Limits of the improved integrability of the volume forms},
   JOURNAL = {Indiana Univ. Math. J.},
  FJOURNAL = {Indiana University Mathematics Journal},
    VOLUME = {44},
      YEAR = {1995},
    NUMBER = {2},
     PAGES = {305--339},
      ISSN = {0022-2518,1943-5258},
   MRCLASS = {46E35 (26B10 30C65 58C35)},
  MRNUMBER = {1355401},
MRREVIEWER = {Juha\ Heinonen},
       DOI = {10.1512/iumj.1995.44.1990},
       URL = {https://doi.org/10.1512/iumj.1995.44.1990},
}

@article {Greco2,
    AUTHOR = {Greco, L.},
     TITLE = {Sharp integrability of nonnegative {J}acobians},
   JOURNAL = {Rend. Mat. Appl. (7)},
  FJOURNAL = {Rendiconti di Matematica e delle sue Applicazioni. Serie VII},
    VOLUME = {18},
      YEAR = {1998},
    NUMBER = {3},
     PAGES = {585--600},
      ISSN = {1120-7183,2532-3350},
   MRCLASS = {42B25 (26B10)},
  MRNUMBER = {1686812},
MRREVIEWER = {Maria\ J.\ Carro},
}

@article {IS,
    AUTHOR = {Iwaniec, T. and Sbordone, C.},
     TITLE = {Weak minima of variational integrals},
   JOURNAL = {J. Reine Angew. Math.},
  FJOURNAL = {Journal f\"ur die Reine und Angewandte Mathematik. [Crelle's
              Journal]},
    VOLUME = {454},
      YEAR = {1994},
     PAGES = {143--161},
      ISSN = {0075-4102,1435-5345},
   MRCLASS = {49K10 (35J99 49N60)},
  MRNUMBER = {1288682},
MRREVIEWER = {Francesco\ Ferro},
       DOI = {10.1515/crll.1994.454.143},
       URL = {https://doi.org/10.1515/crll.1994.454.143},
}

@article {IS2,
    AUTHOR = {Iwaniec, Tadeusz and Sbordone, Carlo},
     TITLE = {On the integrability of the {J}acobian under minimal
              hypotheses},
   JOURNAL = {Arch. Rational Mech. Anal.},
  FJOURNAL = {Archive for Rational Mechanics and Analysis},
    VOLUME = {119},
      YEAR = {1992},
    NUMBER = {2},
     PAGES = {129--143},
      ISSN = {0003-9527},
   MRCLASS = {49Q20 (26B10 46E35 49J10)},
  MRNUMBER = {1176362},
MRREVIEWER = {Paolo\ Marcellini},
       DOI = {10.1007/BF00375119},
       URL = {https://doi.org/10.1007/BF00375119},
}

@article {IV,
    AUTHOR = {Iwaniec, Tadeusz and Verde, Anna},
     TITLE = {A study of {J}acobians in {H}ardy-{O}rlicz spaces},
   JOURNAL = {Proc. Roy. Soc. Edinburgh Sect. A},
  FJOURNAL = {Proceedings of the Royal Society of Edinburgh. Section A.
              Mathematics},
    VOLUME = {129},
      YEAR = {1999},
    NUMBER = {3},
     PAGES = {539--570},
      ISSN = {0308-2105,1473-7124},
   MRCLASS = {42B25 (26B10 46E30)},
  MRNUMBER = {1693625},
MRREVIEWER = {Alberto\ Fiorenza},
       DOI = {10.1017/S0308210500021508},
       URL = {https://doi.org/10.1017/S0308210500021508},
}

@article{Kirszbraun,
    author={Kirszbraun, M. D.},
    title={{\"Uber die zusammenziehende und Lipschitzsche Transformation}},
    journal={Fund. Math.},
    year={1934},
    volume={22},
    issue={1},
    pages={77-108},
    }

@article {Hiro,
    AUTHOR = {Kita, Hiro-o},
     TITLE = {On maximal functions in {O}rlicz spaces},
   JOURNAL = {Proc. Amer. Math. Soc.},
  FJOURNAL = {Proceedings of the American Mathematical Society},
    VOLUME = {124},
      YEAR = {1996},
    NUMBER = {10},
     PAGES = {3019--3025},
      ISSN = {0002-9939,1088-6826},
   MRCLASS = {42B25 (46E30)},
  MRNUMBER = {1376993},
MRREVIEWER = {Yu.\ A.\ Brudny\u i},
       DOI = {10.1090/S0002-9939-96-03807-5},
       URL = {https://doi.org/10.1090/S0002-9939-96-03807-5},
}

@article {KMSX,
    AUTHOR = {Kleiner, Bruce and M\"uller, Stefan and Sz\'ekelyhidi, Jr.,
              L\'aszl\'o{} and Xie, Xiangdong},
     TITLE = {Rigidity of {E}uclidean product structure: breakdown for low
              {S}obolev exponents},
   JOURNAL = {Commun. Pure Appl. Anal.},
  FJOURNAL = {Communications on Pure and Applied Analysis},
    VOLUME = {23},
      YEAR = {2024},
    NUMBER = {10},
     PAGES = {1569--1607},
      ISSN = {1534-0392,1553-5258},
   MRCLASS = {35R70 (30C62 30C65 35A35 46E35)},
  MRNUMBER = {4799456},
       DOI = {10.3934/cpaa.2024029},
       URL = {https://doi.org/10.3934/cpaa.2024029},
}

@article {Lewis,
    AUTHOR = {Lewis, John L.},
     TITLE = {On very weak solutions of certain elliptic systems},
   JOURNAL = {Comm. Partial Differential Equations},
  FJOURNAL = {Communications in Partial Differential Equations},
    VOLUME = {18},
      YEAR = {1993},
    NUMBER = {9-10},
     PAGES = {1515--1537},
      ISSN = {0360-5302,1532-4133},
   MRCLASS = {35J60 (46N20)},
  MRNUMBER = {1239922},
MRREVIEWER = {Nicoletta\ Anna\ Tchou},
       DOI = {10.1080/03605309308820984},
       URL = {https://doi.org/10.1080/03605309308820984},
}

@article{Liu,
    author={Liu, F. C.},
    title={{A Luzin type property of Sobolev functions}},
    journal={Indiana Univ. Math. J.},
    volume={26},
    number={4},
    year={1977},
    pages={645-651},
    }

@article{Meyers,
author = {Meyers, Norman G.},
journal = {Annali della Scuola Normale Superiore di Pisa - Classe di Scienze},
language = {eng},
number = {3},
pages = {189-206},
publisher = {Scuola normale superiore},
title = {{An $L^p$-estimate for the gradient of solutions of second order elliptic divergence equations}},
url = {http://eudml.org/doc/83302},
volume = {17},
year = {1963},
}

@article{Mueller,
author = {Stefan M{\"u}ller},
title = {{A surprising higher integrability property of mappings with positive determinant}},
volume = {21},
journal = {Bulletin (New Series) of the American Mathematical Society},
number = {2},
publisher = {American Mathematical Society},
pages = {245 -- 248},
year = {1989},
}

@article{Mueller2,
author = {M\"uller, Stefan},
journal = {Journal für die reine und angewandte Mathematik},
pages = {20-34},
title = {Higher integrability of determinants and weak convergence in L1.},
url = {http://eudml.org/doc/153276},
volume = {412},
year = {1990},
}

@article{Raita25,
author = {Rai\c{t}\u{a}, B., Bogdan},
year = {2025},
month = {02},
pages = {},
title = {Quasiconvexity and self-improving size estimates},
journal = {arXiv:2502.11980}
}

@article{Sbordone96,
author = {Carlo, Sbordone},
year = {1996},
month = {10},
pages = {},
title = {Grand Sobolev spaces and their application to variational problems},
volume = {51},
journal = {Le Matematiche}
}

@article{Whitney,
	author={Whitney,H.},
	title={{Analytic Extensions of Differentiable Functions Defined in Closed Sets}},
	journal={Trans. Am. Math. Soc.},
	volume={36},
	year={1934},
	pages={63-89}
	}

@article {Zhang,
    AUTHOR = {Zhang, Kewei},
     TITLE = {A construction of quasiconvex functions with linear growth at
              infinity},
   JOURNAL = {Ann. Scuola Norm. Sup. Pisa Cl. Sci. (4)},
  FJOURNAL = {Annali della Scuola Normale Superiore di Pisa. Classe di
              Scienze. Serie IV},
    VOLUME = {19},
      YEAR = {1992},
    NUMBER = {3},
     PAGES = {313--326},
      ISSN = {0391-173X,2036-2145},
   MRCLASS = {49J45 (26B25 90C25)},
  MRNUMBER = {1205403},
MRREVIEWER = {John\ M.\ Ball},
       URL = {http://www.numdam.org/item?id=ASNSP_1992_4_19_3_313_0},
}
\bibliographystyle{abbrv}

\end{document}